\documentclass[a4paper,11pt,DIV=11,abstract=on]{scrartcl}

\usepackage[T1]{fontenc}
\usepackage[utf8]{inputenc}
\usepackage{amssymb,amsmath,amsthm,bm,mathtools}
\usepackage{aliascnt}
\usepackage{array}
\mathtoolsset{showonlyrefs}
\usepackage{csquotes}
\usepackage{booktabs}
\usepackage[shortlabels]{enumitem}
\usepackage{cite}

\usepackage{graphicx}
\usepackage{xcolor}
\graphicspath{{../images/},{../matlab/images/}}
\usepackage{subcaption}
\usepackage{tikz,pgfplots}
\pgfplotsset{compat=1.18}

\usepackage{bbm} 

\usepackage[draft, todonotes={textsize=scriptsize}]{changes}
\setuptodonotes{color=red, backgroundcolor=red!40!white,
  bordercolor=red, tickmarkheight=10pt}
\definechangesauthor[name=Robert, color=cyan]{RB}
\definechangesauthor[name=Michael, color=violet]{MQ}

\usepackage[english,pdfusetitle,colorlinks=true,linkcolor=blue,citecolor=green!50!black,
bookmarks=true]{hyperref}

\hypersetup{
  pdfauthor={
    Michael Quellmalz, quellmalz@math.tu-berlin.de;
    Robert Beinert, beinert@math.tu-berlin.de;
    Gabriele Steidl, steidl@math.tu-berlin.de
  }
}

\newtheorem{theorem}{Theorem}[section]

\newaliascnt{lemma}{theorem}

\aliascntresetthe{lemma}

\newaliascnt{corollary}{theorem}
\newtheorem{corollary}[corollary]{Corollary}
\aliascntresetthe{corollary}

\newaliascnt{proposition}{theorem}
\newtheorem{proposition}[proposition]{Proposition}
\aliascntresetthe{proposition}

\theoremstyle{definition}

\newtheorem{remark}[theorem]{Remark}
\newtheorem*{remark*}{Remark}

\newaliascnt{algorithm}{theorem}
\newtheorem{algorithm}[algorithm]{Algorithm}
\aliascntresetthe{algorithm}

\newcommand{\N}{\ensuremath{\mathbb{N}}}
\newcommand{\NN}{\ensuremath{\mathbb{N}_0}}
\newcommand{\T}{\ensuremath{\mathbb{T}}}
\renewcommand{\S}{\ensuremath{\mathbb{S}}}
\newcommand{\M}{\ensuremath{\mathcal{M}}}

\newcommand{\cS}{\ensuremath{\mathcal{S}}}
\newcommand{\cT}{\mathcal{T}}
\newcommand{\cA}{\mathcal{A}}
\newcommand{\cZ}{\mathcal{Z}}

\renewcommand{\P}{\ensuremath{\mathcal{P}}}
\newcommand{\ac}{\ensuremath{\mathrm{ac}}}
\newcommand{\Pac}{\ensuremath{\mathcal{P}_{\ac}}}

\newcommand{\V}{\mathcal{V}}
\newcommand{\W}{\mathcal{W}}
\newcommand{\Z}{\ensuremath{\mathbb{Z}}}

\newcommand{\R}{\ensuremath{\mathbb{R}}}

\newcommand{\X}{\ensuremath{\mathbb{X}}}
\newcommand{\Y}{\ensuremath{\mathbb{Y}}}
\newcommand{\II}{\ensuremath{\mathbb{I}}}
\newcommand{\SO}{\ensuremath{\mathrm{SO}(3)}}

\newcommand{\abs}[1]{\ensuremath{\left\vert#1\right\vert}}
\newcommand{\inn}[1]{\ensuremath{\left\langle#1\right\rangle}}
\newcommand{\dx}{\mathrm{d}}
\newcommand{\ds}{\mathrm{ds}}
\newcommand{\e}{\mathrm{e}}

\renewcommand{\i}{\mathrm{i}}
\newcommand{\zb}[1]{\ensuremath{\boldsymbol{#1}}}

\DeclareMathOperator{\azi}{azi}
\DeclareMathOperator{\zen}{zen}

\DeclareMathOperator*{\argmin}{arg\,min}

\DeclareMathOperator*{\esssup}{ess\,sup}

\DeclareMathOperator{\sym}{sym}

\DeclareMathOperator{\proj}{proj}
\DeclareMathOperator{\prox}{prox}

\DeclareMathOperator{\Id}{Id}
\DeclareMathOperator{\KL}{KL}
\DeclareMathOperator{\sph}{\Phi}
\DeclareMathOperator{\eul}{\Psi}
\DeclareMathOperator{\VSW}{VSW}
\DeclareMathOperator{\SSW}{SSW}
\DeclareMathOperator{\WS}{W}

\DeclareMathOperator{\CDT}{CDT}
\DeclareMathOperator{\sCDT}{cCDT}

\renewcommand{\d}{\, \mathrm{d}}
\renewcommand{\Box}{\hspace*{0ex} \hfill \rule{1.5ex}{1.5ex} \\}

\newcommand{\bigo}{\ensuremath{\mathcal{O}}}

\newcommand{\norm}[1]{\left\lVert #1
  \right\rVert}

\newcommand{\bg}{{\boldsymbol g}}

\newcommand{\bx}{{\boldsymbol x}}
\newcommand{\by}{{\boldsymbol y}}

\newcommand{\bw}{{\boldsymbol w}}

\newcommand{\bQ}{{\boldsymbol Q}}

\newcommand{\bxi}{\boldsymbol\xi}

\mathtoolsset{showonlyrefs}
\newcommand{\tT}{\top}

\begin{document}

\def\sectionautorefname{Section}
\def\subsectionautorefname{Section}

\title{Sliced Optimal Transport on the Sphere}

\author{
	Michael Quellmalz\textsuperscript{1,2}
	\and
	Robert Beinert\textsuperscript{1,2}
	\and
	Gabriele Steidl\textsuperscript{1}
}

\maketitle
\date{\today}

\footnotetext[1]{Institute of Mathematics,
	Technische Universität Berlin,
	Stra{\ss}e des 17.\ Juni 136, 
	10623 Berlin, Germany,
	\ttfamily{\{quellmalz, beinert, steidl\}@math.tu-berlin.de}
    \url{http://tu.berlin/imageanalysis}
	} 

\footnotetext[2]{These authors have contributed equally to the work.}

\begin{abstract}
Sliced optimal transport reduces
optimal transport on multi-dimensional domains
to transport on the line.
More precisely,
sliced optimal transport is the concatenation of 
the well-known Radon transform and
the cumulative density transform,
which analy\-ti\-cal\-ly yields the solutions of the reduced transport problems.
Inspired by this concept,
we propose two adaptions for optimal transport on the 2-sphere.
Firstly, 
as counterpart to the Radon transform, 
we introduce the vertical slice transform,
which integrates along all circles orthogonal to a given direction.
Secondly,
we introduce a semicircle transform,
which integrates along all half great circles with an appropriate weight function.
Both transforms are generalized to arbitrary measures on the sphere.
While the vertical slice transform can be combined with optimal transport on the interval and
leads to a sliced Wasserstein distance restricted to even probability measures,
the semicircle transform is related to optimal transport on the circle
and results in a different sliced Wasserstein distance for arbitrary probability measures.
The applicability of both novel sliced optimal transport concepts on the sphere
is demonstrated by proof-of-concept examples 
dealing with the interpolation and classification of spherical probability measures.
The numerical implementation
relies on the singular value decompositions of both transforms 
and
fast Fourier techniques.
For the inversion with respect to probability measures,
we propose the minimization of an entropy-regularized Kullback--Leibler divergence,
which can be numerically realized using a primal-dual proximal splitting algorithm.
\end{abstract}

\section{Introduction}
Optimal transport and in particular Wasserstein distances between measures have 
received much attention from a theoretical and practical 
point of view \cite{PeyCut19,San15,Vil03}
and recently became of interest in neural gradient flows \cite{FZTC2022,KSB2022,AHS2023}.
While Wasserstein distances are in general hard to compute, there exist
analytic formulas for optimal transport on the line.
Therefore sliced Wasserstein distances, 
which basically combine the Radon transform in Euclidean spaces
with optimal transport on the line, have become quite popular \cite{RabPeyDelBer12,San15,NRNRNH23}.
In particular, the related	Radon cumulative distribution transform has been applied for interpolation and classification as well as for model reduction \cite{KolParRoh16,GuaLiaDuYin19,RenWolMao21,ShiRub21,BRPP15}. 
The idea behind sliced optimal transport has been generalized and transferred to many related problems.
There exists sliced variants \cite{BC19,BSTK23} of partial optimal transport \cite{Fig10,CAG20},
where only a fraction of mass is transported,
and a sliced version \cite{CTP+21}
of multi-marginal optimal transport \cite{GS98,BLNS22,BBS22},
considering the transport between several measures
instead of only two.
For optimal transport on Riemannian manifolds,
sliced Wasserstein distances based on the push-forward of the eigenfunctions 
of the Laplacian have been proposed in \cite{RusMaj20}.
Especially for shape and graph analysis,
sliced optimal transport has been transferred to
the Gromov--Wasserstein setting \cite{VFC+19}, 
which more generally defines a metric 
between metric measure spaces \cite{Mem11,Stu20,BeBeSt22}.
Differently from the Wasserstein formulation with its analytic solution,
the Gromov--Wasserstein transport on the line is more involved \cite{BCS22,DLV23}.

In this paper, 
we transfer the slicing approach to
optimal transport on the two-dimensional sphere.
Spherical optimal transport has been intensely studied in recent years.
For instance,
the problem can be solved using a Monge--Ampère type equation \cite{HamTur22,McrCotBud18,WelBroBudCul16}
or a variational framework \cite{CuiQi19}.
The regularity of optimal maps has been investigated in \cite{Loe10}.
Spherical Wasserstein barycenters have been computed 
using a stochastic projected subgradient method \cite{StaClaSolJeg17} 
and have been estimated on random graphs \cite{TheKer22}.

To introduce slicing frameworks on the sphere,
we do not follow the Laplacian approach in \cite{RusMaj20},
but focus on spherical counterparts of the Radon transform.
A well-known one is the Funk--Radon transform \cite{Fun13,Hel11,LoRiSpSp11,QueWeiHubErc23}, which takes integrals along all great circles.
Integration along all circles of a fixed radius were studied in \cite{Sch69,Rub00}.
Further Radon-type transforms were considered based on intersections with planes containing a fixed point inside the sphere \cite{Pal16,Sal17,Qu17,Que20},
on the sphere \cite{AbDa93,Rub22},
and outside the sphere \cite{AgRu20}.
Moreover, 
transforms including derivatives were proposed in \cite{MaMaOd01,QuHiLo18}.
However, in the context of sliced optimal transport,
we require that
probability density functions on the sphere are mapped
to a family of probability density functions on one-dimensional domains.
For this purpose, 
we consider two specific spherical transforms, namely 
the vertical slice transform
and the normalized semicircle transform.

\begin{figure}
  \centering
  \begin{subfigure}[t]{.49\textwidth} \centering
  \includegraphics[width=.65\textwidth]{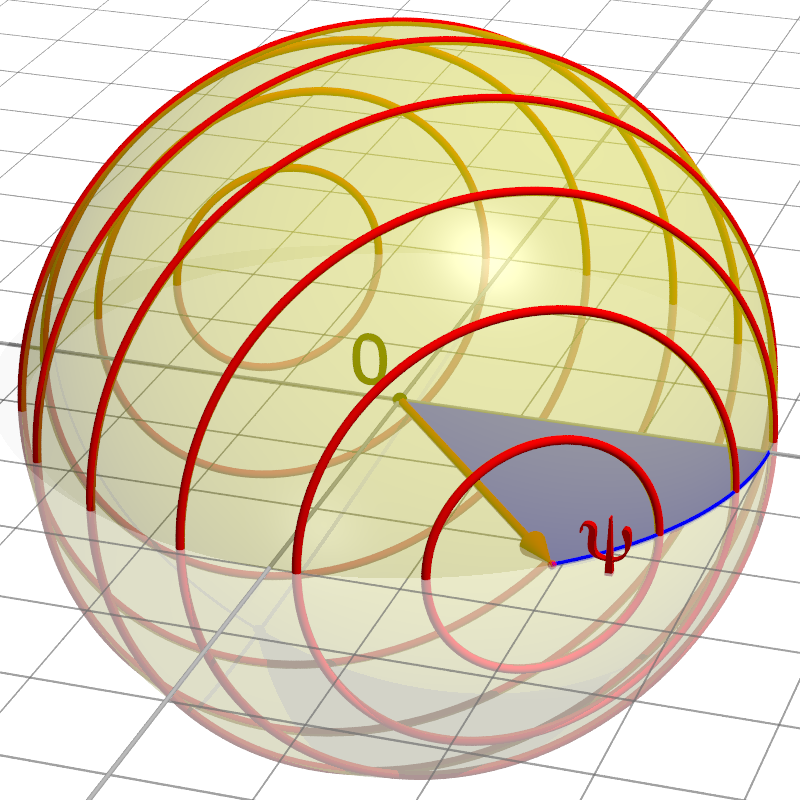}  
  \caption{
    Vertical slices for a fixed direction.}
     \label{fig:vertical_slice_1} 
   \end{subfigure}
  \begin{subfigure}[t]{.49\textwidth} \centering
  \includegraphics[width=.65\textwidth]{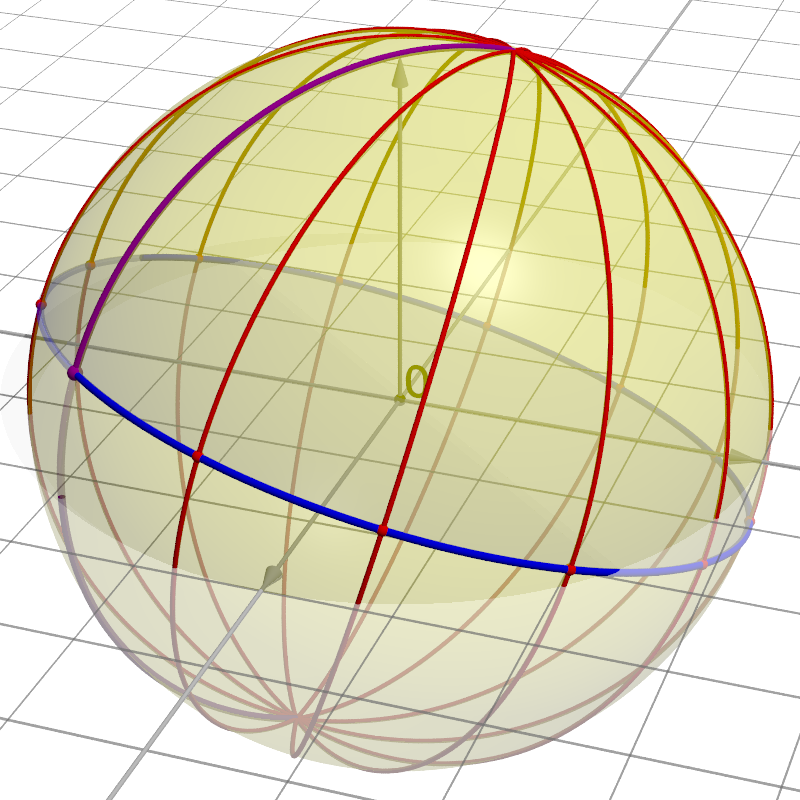}  
  \caption{
    Semicircles (red) starting in a fixed point. 
    }
   \label{fig:vertical_slice_2} 
   \end{subfigure}
   \caption{Areas of integration of the spherical transforms.}
\end{figure}

The vertical slice transform was first considered in \cite{GiReSh94} and applied in \cite{ZaSc10,HiQu15circav} for photoacoustic tomography.
The generalization to higher dimensions is due to \cite{Rub19}.
The basic idea is to take means along parallel circles, 
see \autoref{fig:vertical_slice_1},
which gives a probability density function on an interval. 
The process is then repeated for further directions.
Geometrically, the areas of integration for a fixed direction can be imagined like an \enquote{egg cutter} applied to the sphere.
We generalize the vertical slice transform to probability measures and use it to define a vertical sliced Wasserstein distance.
Radon transforms of measures have been considered in the context of a dual fibration,
cf.\ \cite{GelShm83,Pal95} and \cite[Chap.\ 2, § 2]{Hel11},
where they are defined via duality.
We will see that our definition via the push-forward of measures
can be also derived from that point of view.

An (unnormalized) semicircle transform 
was examined in \cite{Gro98,HiPoQu18}.
It takes integrals along semicircles starting in a fixed point,
see \autoref{fig:vertical_slice_2},
and yields a function defined on the one-dimensional unit circle.
The process is then repeated for further starting points.
This transform has been combined with optimal transport on the circle
to obtain a sliced Wasserstein distance \cite{Bon22}.
However, 
the crucial point is here that
the unnormalized semicircle transform does not 
map probability density functions
to probability density functions,
meaning that optimal transport techniques on the circle
cannot be applied.
In the numeric part of \cite{Bon22},
the authors restrict themselves to point measures, 
which then are projected onto great circles.
This approach corresponds to 
an appropriately \emph{normalized} semicircle transform instead, where the integrand is multiplied with a certain weight function.
In this paper, 
we introduce and study this normalized semicircle transform
in a rigours manner to obtain a semicircular sliced Wasserstein distance.

\paragraph{Main contributions}
\begin{itemize}
    \item We give rigorous definitions of 
    the vertical slice and the normalized semicircle transform, 
    which are originally considered only for functions,
    and generalize them to measures using an appropriate push-forward.
    For absolutely continuous measures,
    the generalized and initial definitions coincide in the sense that
    merely the density function has to be transformed.
    Furthermore, probability measures are transformed to probability measures.
    \item We prove a singular value decomposition of the normalized semicircle transform, which provides an approach for numerical computations. 
    Moreover, the singular value decompositions of the vertical slice and the
    normalized semicircle transform allow the inversion via their Moore--Penrose pseudoinverses.
    \item We define sliced Wasserstein distances on the sphere based on both transform.
    We show that the normalized semicircle transform is injective for all measures,
    and hence the sliced Wasserstein distance indeed fulfills the properties of a metric.
    Furthermore, the vertical sliced Wasserstein distance is a metric for even measures on the sphere.
    \item We propose a Tikhonov-type regularization which minimizes a variational model consisting of the entropy-regularized Kullback--Leibler divergence.
    This ensures that the inverse is a probability measure and in particular non-negative.
    Further, this allows to compute a sliced CDT interpolation between spherical probability measures to approximate Wasserstein barycenters.
\end{itemize}

\paragraph{Outline of the Paper}
We start in \autoref{sec:prelim} with the necessary preliminaries on optimal transport, 
the unit sphere and the rotation group on $\R^3$.
Then, 
we introduce the two counterparts of the Radon transform on $\S^2$, namely the
vertical slice transform in \autoref{sec:vert_slice} and the normalized semicircle transform in \autoref{sec:semicircle}.
First, we define the transforms for functions and
derive their adjoint operators and
singular value decompositions on $L^2(\S^2)$. 
In order to combine these transforms with optimal transport on the interval and the circle respectively,
we have to enlarge their definitions to measure spaces 
which we have not found in a mathematically rigorous form in the literature.
\autoref{sec:SOT} connects the above transforms with optimal transport to
introduce spherical sliced Wasserstein distances
for measures on the sphere.
\autoref{sec:disc-and-inv} deals with the discretization of the spherical transforms and their inversion, which is an ill-posed problem.
For an approximate inversion, we can use the truncated Moore--Penrose pseudoinverse.
However, when dealing with probability density functions, this inversion does not guarantee the non-negativity of
the reconstructed function.
Therefore, we suggest another reconstruction 
which minimizes a variational model consisting of an entropy-regularized Kullback--Leibler divergence, see \autoref{sec:reg}. 
The actual minimization can be done by a primal-dual splitting.
Numerical proof-of-concept results are reported in \autoref{sec:numerics},
where we provide two kinds of experiments.
First, we show in \autoref{sec:int-meas} that Wasserstein barycenters on the sphere
can be  approximated using sliced Wasserstein transforms
and Wasserstein interpolation on the interval and the circle respectively.
These results require in particular the inversion of the 
sliced spherical transforms.
Second, we demonstrate by a synthetic example that the binary classification of different measures  
is in principle possible in \autoref{sec:classification}.

\section{Preliminaries} \label{sec:prelim}
In this section, we first provide the notation and necessary preliminaries 
on optimal transport, in particular on the interval and the circle.
Then, we recall basic facts about the unit sphere and the rotation group on $\mathbb R^3$. 

\subsection{Measures and Optimal Transport}
Let $\X$ be a compact metric space with metric $d\colon \X \times \X \to \R$,
and let $\mathcal B(\X)$ be the Borel $\sigma$-algebra induced by $d$.
By $\M(\X)$,
we denote the Banach space of signed, finite measures, and 
by $\P(\X)$ the subset of probability measures
on $\X$. 
The pre-dual space of $\M(\X)$ is $C(\X)$.
Let $\Y$ be another compact metric space and $T\colon \X \to \Y$
be measurable. For $\mu \in \M(\X)$, 
we define the \emph{push-forward measure} 
$T_\# \mu \coloneqq \mu \circ T^{-1} \in \M(\mathbb Y)$.
For any measure $\pi \in \M(\X \times \Y)$ 
with first marginal $\mu \in \M(\X)$,
i.e.,
$\pi(B \times \Y) = \mu(B)$ for all $B \in \mathcal B(\X)$,
we call a collection of measures $\pi_x \in \M(\Y)$, $x \in \X$,
a \emph{disintegration family} if
\begin{equation}
    \int_{\X \times \Y} f(x,y) \d \pi(x,y)
    =
    \int_\X \int_\Y f(x,y) \d \pi_x(y) \d \mu(x)
\end{equation}
for all measurable functions $f$ on $\X \times \Y$.

The $p$-\emph{Wasserstein distance}, $p\in[1,\infty)$, of $\mu,\nu \in \mathcal P(\mathbb X)$ is given by
\begin{equation} \label{eq:Wp}
    \WS_p^p(\mu,\nu)
    \coloneqq
    \min_{\pi\in\Pi(\mu,\nu)} \int_{\mathbb X^2} d^p(x,y) \d \pi (x,y),
\end{equation}
with
$
\Pi(\mu,\nu) \coloneqq \{\pi \in\M(\mathbb X \times \mathbb X): 
\pi(B \times \mathbb  X) = \mu(B),
\pi(\mathbb  X\times B) = \nu(B)
\ \text{for all }  B \in \mathcal B(\mathbb X)\}
$. 
It defines a metric on $\mathcal P(\mathbb X)$.
The metric space $\mathcal P^p(\mathbb X) \coloneqq (\mathcal P(\mathbb X),W_p)$
is called $p$-Wasserstein space and, in case $p=2$, just Wasserstein space.
The above Wasserstein distance is just a special case of the more general optimal transport problem,
where $d^p(x,y)$ can be replaced by a more general cost function $c(x,y)$.
For $\delta\in[0,1]$, 
the \emph{$p$-Wasserstein barycenter} between $\mu,\nu\in\P^p(\X)$ is the minimizer of 
\begin{equation} \label{eq:W-bary}
    \min_{\omega\in\P(\X)}
    (1-\delta)\, \WS_p^p(\mu,\omega)
    +
    \delta\, \WS_p^p(\nu,\omega),
\end{equation}
see \cite{AguCar11}.
Note that the Wasserstein barycenter between absolutely continuous measures is unique,
cf.\ \cite{KimPas17}.

\paragraph{Optimal Transport on the Interval}

If $\X$ is the \emph{unit interval} $\II \coloneqq [-1,1]$ with the distance $d(x,y) = \abs{x-y}$, the optimal transport between two probability measures $\mu,\nu\in \mathcal P(\II)$ can be computed easily \cite{PeyCut19,San15,Vil03} 
using the \emph{cumulative distribution function} 
$F_\mu(x) \coloneqq \mu([-1,x])$, $x\in\II$,
which is non-decreasing and right-continuous.
Its pseudoinverse, 
the \emph{quantile function}
$
F_\mu^{-1}(r) \coloneqq \min\{ x\in\II: F_\mu(x)\ge r\}$,
$r\in[0,1]$,
is non-decreasing and left-continuous.
The measure $\mu$ can be recovered by $\mu = (F_\mu^{-1})_\# \sigma_{[0,1]}$,
where $\sigma_{[0,1]}$ denotes the Lebesque measure on $[0,1]$.
The $p$-Wasserstein distance \eqref{eq:Wp} between $\mu,\nu \in \P^p(\II)$
now equals 
$\WS_p(\mu,\nu)
=
\lVert F_\mu^{-1} - F_\nu^{-1} \rVert_{L^p([0,1])}$.
Moreover,
if $\mu \in \Pac(\II)$,
where $\Pac(\II)$ denotes the probability measures that are
\emph{absolutely continuous} with respect to the Lebesgue measure,
then the optimal transport plan $\pi$ in \eqref{eq:Wp}
is uniquely given by 
\begin{equation}  
    \pi = (\Id, T^{\mu,\nu})_\# \mu
    \quad\text{with}\quad 
    T^{\mu,\nu}(x) \coloneqq F_\nu^{-1}(F_\mu(x)),
    \quad
    x \in \II.
\end{equation}

Based on the \emph{optimal transport map} $T^{\mu,\nu}$,
the Wasserstein space $\P^p(\II)$ can be isometrically embedded 
into $L^p_\omega(\II)$ with $\omega \in \Pac(\II)$
\cite{KolParRoh16, PaKoSo18, BeBeSt22},
where $L^p_\omega(\II)$ consists of all $p$-integrable functions with respect to $\omega$.
More precisely,
for the reference measure $\omega \in \Pac(\II)$,
the \emph{cumulative distribution transform} (CDT)
is defined by 
$\CDT_\omega \colon \P^p(\II) \to L^p_\omega(\II)$ with
\begin{equation} \label{eq:cdt}
    \CDT_{\omega}[\mu] (x) 
    \coloneqq 
    (T^{\omega,\mu} - \Id)(x)
    =
    \bigl(F_{\mu}^{-1} \circ F_{\omega} \bigr) (x) - x,
    \quad
    x\in\II,
\end{equation}
and we especially have
$\WS_p(\mu,\nu) = \lVert \CDT_\omega[\mu] - \CDT_{\omega}[\nu] \rVert_{L^p_\omega(\II)}$.
The CDT is in fact a mapping from $\P^p(\II)$ into the tangent space of $\P^p(\II)$ at $\omega$,
see \cite[§~8.5]{AGS05}.
Due to the relation to the optimal transport map,
the CDT can be inverted by $\mu = \CDT^{-1}_\omega[h] \coloneqq (h + \Id)_\# \omega$ 
{for $h = \CDT_\omega[\mu]$}. 
If $\mu,\omega \in \Pac(\II)$ possess the density functions $f_\mu > 0$, $f_\omega > 0$,
then, by the transformation formula for push-forward measures, $f_\mu$ can be recovered by
\begin{equation} \label{eq:icdt}
        f_\mu(x)
        =
        \left( g^{-1} \right)'(x)\, f_\omega(g^{-1}(x))
        \quad
        \text{with} \quad
        g(x) = \CDT_\omega[\mu](x) + x
        ,\quad
        x \in \II.
\end{equation}
For $\mu,\nu \in \P(\II)$,
and an arbitrary reference measure $\omega \in \Pac(\II)$,
the 2-Wasserstein barycenter \eqref{eq:W-bary} has the form \begin{equation}
    \label{eq:cdt-bary}
    \CDT^{-1}_{\omega}\left(
    \delta \CDT_\omega [\nu] + (1-\delta) \CDT_\omega[\mu] \right),
\end{equation}
see \cite{KolParRoh16}.
In particular for $\omega = \mu$,
we have by \eqref{eq:cdt} that $\CDT_\mu[\mu](x) = F_\mu^{-1}(F_\mu(x))-x = 0$ and therefore the barycenter \eqref{eq:cdt-bary} becomes
\begin{equation}    \label{eq:cdt-bary_1}
    \CDT^{-1}_{\mu}\left(
    \delta \CDT_\mu [\nu] \right).
\end{equation}

\paragraph{Optimal Transport on the Circle}

On the \emph{circle} $\T \coloneqq \R/(2\pi\Z)$ 
equipped with the metric $d(x,y) \coloneqq \min_{k\in\Z} \abs{x-y+2\pi k}$,
the optimal transport can be computed in a similar manner
by incorporating the periodicity. 
Following \cite{DelSalSob10,RabDelGou11},
we define the \emph{(extended) cumulative distribution function} by
$\tilde F_\mu(x) \coloneqq \mu([0,x])$ for $x \in [0,2\pi]$
and extend it to $\mathbb R$ by the convention $\tilde F_\mu(x + 2\pi) \coloneqq \tilde F_\mu(x) + 1$.
Its pseudoinverse, 
the \emph{(extended) quantile function},
is defined as
$\tilde F_\mu^{-1}(r) \coloneqq \min\{ x\in\R: \tilde F_\mu(x)\ge r\}$ for
$r\in \R$. 
Note that $\tilde F$ and $\tilde F^{-1}$ are mappings defined on entire $\R$.
The $p$-Wasserstein distance between $\mu,\nu \in \P(\T)$ is given by
\begin{equation} \label{eq:Wp-circle}
    W_p^p(\mu,\nu)
    =
    \min_{\theta\in\R} 
    \int_{0}^{1} \lvert\tilde F_\mu^{-1}(r) - (\tilde F_\nu - \theta)^{-1}(r) \rvert^p \d r,
\end{equation}
where $(\tilde F_\nu - \theta)^{-1}$ is the pseudoinverse of the shifted cumulative distribution function
\cite{RabDelGou11}.
For $\mu \in \Pac(\T)$,
each minimizer $\theta$ of \eqref{eq:Wp-circle}
yields an optimal transport plan
\begin{equation}
    \pi = (\Id, \iota(\tilde T^{\mu,\nu}))_\# \mu
    \quad\text{with}\quad
    \tilde T^{\mu,\nu}(x) \coloneqq (\tilde F_\nu - \theta)^{-1} (\tilde F_\mu(x)),
    \quad x \in [0,2\pi),
\end{equation}
where $\iota \colon \R \to \T$ denotes the canonical projection from the line to the circle.
Note that $\tilde T^{\mu,\nu}(x) \in \R$ is the representative of 
$\iota(\tilde T^{\mu,\nu}(x)) \in \T$ with
\begin{equation}
    d(x, \iota(\tilde T^{\mu,\nu}(x)))
    =
    \lvert x - \tilde T^{\mu,\nu}(x) \rvert.
\end{equation}
If $p>1$ and $\mu,\nu\in\Pac(\T)$,
the minimizer $\theta$ of \eqref{eq:Wp-circle} is unique. This follows by the proof of \cite[Lem.~5.2]{DelSalSob10}, where it is shown that the objective of \eqref{eq:Wp-circle} is convex in $\theta$, but the argument even implies strict convexity.
In analogy to \eqref{eq:cdt}, we define
for $p\in(1,\infty)$, the \emph{circular CDT} (cCDT) of $\mu\in\Pac(\T)$ with reference measure $\omega\in\Pac(\T)$ by
$\sCDT_\omega \colon \P^p(\T) \to L^p_\omega(\T)$ with
\begin{equation} 
    \label{eq:cdt-tc}
    \sCDT_{\omega}[\mu] (x) 
    \coloneqq 
    (\tilde T^{\omega, \mu} - \Id)(x)
    =
    \bigl((\tilde F_{\mu}-\theta_{\omega,\mu})^{-1} \circ \tilde F_{\omega}\bigr) (x) - x,
    \quad x\in[0,2\pi),
\end{equation}
where $\tilde T^{\omega,\mu}$ is the optimal transport plan and 
$\theta_{\omega,\mu}$
the minimizer of \eqref{eq:Wp-circle}.
Note that the cCDT is no longer an isometric embedding.
The cCDT can be inverted by 
$\mu = \sCDT^{-1}_\omega[h] 
\coloneqq (\iota \circ(h + \Id))_\#\omega$ for $h = \sCDT_\omega[\mu]$.
If $\mu,\omega \in \Pac(\T)$ have densities $f_\mu > 0$, $f_\omega > 0$,
the density $f_\mu$ can be recovered similarly to \eqref{eq:icdt} via
\begin{equation} \label{eq:iccdt}
        f_\mu(x)
        =
        \left( g^{-1}\right)'(x)\, f_\omega(g^{-1}(x))
        \quad
        \text{with} \quad
        g(x) = \iota \left(\sCDT_\omega[\mu](x) + x \right)  ,\quad
        x \in \T.
\end{equation}
In analogy to \eqref{eq:cdt-bary} with $\omega = \mu$,
we interpolate between the measures $\mu, \nu \in \Pac(\T)$ by
\begin{equation} \label{eq:W2-bary-circle}
    \sCDT^{-1}_{\mu}[\delta \sCDT_{\mu}[\nu]].
\end{equation}

\subsection{Sphere and Rotation Group}
\paragraph{Unit  Sphere}
The two-dimensional unit sphere is defined as
$\S^{2} \coloneqq \{\bx \in\R^3: \norm{\bx}=1\}$.
The canonical unit vectors are henceforth denoted by $\zb e^j$,
$j=1,2,3$.
Points $\bxi \in \S^2$ can be parameterized in 
\emph{spherical coordinates} 
\begin{equation} \label{eq:xi}
  \bxi = \sph(\varphi,\vartheta)
  \coloneqq
  \left( \cos\varphi \,\sin\vartheta,\,
  \sin\varphi \,\sin\vartheta ,\,
  \cos\vartheta\right) \in \S^2
  ,\quad \varphi\in\T,\ \vartheta\in[0,\pi].
\end{equation}
The restriction
$\sph\colon \left(\T\times (0,\pi) \right) \cup \left(\{0\}\times\{0,\pi \} \right) \to \mathbb S^2$ is a bijective mapping.
We denote the first and second component of this restriction as \emph{azimuth angle} $\azi(\bxi)$ and \emph{zenith angle} $\zen(\bxi)$, respectively,
which are uniquely given by 
\begin{equation} \label{eq:phi}
  \azi(\sph(\varphi,\vartheta)) = \varphi
  \quad\text{and}\quad
  \zen(\sph(\varphi,\vartheta)) = \vartheta
\end{equation}
for all $(\varphi,\vartheta) \in \left(\T\times (0,\pi) \right) \cup \left(\{0\}\times\{0,\pi \} \right)$.
The \emph{surface measure} $\sigma_{\S^2}$ on the sphere is given by
\begin{equation} \label{eq:int_S}
    \int_{\S^2} f(\bxi) \d\sigma_{\S^2}(\bxi)
    =
    \int_{0}^{\pi} \int_{\T} f(\sph(\varphi,\vartheta))\, \sin\vartheta \d\varphi \d \vartheta.
\end{equation}
Normalizing $\sigma_{\S^2}$
yields the \emph{uniform measure} $u_{\S^2} \coloneqq (4\pi)^{-1} \sigma_{\S^2}$.
We denote by $L^{p}(\S^2)$, $p \in [1,\infty]$, the Banach space
of all (equivalence classes of) $p$-integrable functions on $\S^2$,
where we use the above surface measure.

We define the \emph{spherical harmonics} of degree $n\in\NN$ and order $k=-n,\dots,n$ by
\begin{equation} \label{eq:Y}
    Y_n^k(\sph(\varphi,t))
    \coloneqq
    \sqrt{\frac{2n+1}{4\pi} \frac{(n-k)!}{(n+k)!}}\, P_n^k(\cos\vartheta)\, \e^{\i k \varphi},
\end{equation}
where $P_n^k\colon [-1,1] \to \R$ denotes the \emph{associated Legendre functions}
defined by
\begin{equation} \label{eq:Pnk}
    P_n^k(t)
    \coloneqq
    \frac{(-1)^k}{2^n n!} (1-t^2)^{\frac{k}{2}} 
		\frac{ \text{d}^{n+k} (t^2-1)^n}{ \text{ d} t^{n+k}}
    ,\qquad \ n\in\N_0,\ k\in\{0,\dots,n\}
\end{equation}
and 
\begin{equation} \label{eq:Pn-k}
P_n^{-k} 
\coloneqq
(-1)^k \frac{(n-k)!}{(n+k)!} P_n^k .
\end{equation}
The spherical harmonics $\{Y_n^k:  n\in\NN,\ k= -n, ,\dots,n\}$ form an orthonormal basis
of $L^2(\mathbb S^2)$.
Finally, the \emph{Sobolev space} $H^s(\S^2)$ with  $s\ge 0$,
is defined as the completion of $C^\infty(\S^{2})$ with respect to the norm 
    \begin{equation} \label{eq:Hs-norm}
        \norm{f}_{H^s(\S^{2})}
        \coloneqq
        \sum_{n=0}^{\infty} \left(n+\tfrac{1}{2}\right)^{2s}
        \sum_{k=-n}^{n}
         \abs{\langle f,Y_n^k\rangle_{L^2(\S^{2})} }^2.
    \end{equation}

\paragraph{Rotation Group}
Next, we are interested in the \emph{rotation group}
$$\SO \coloneqq \{\bQ\in\R^{3\times3}: \bQ^\top \bQ=I, \det(\bQ)=1\}.$$
Any matrix in $\SO$ has an \emph{Euler angle} parameterization
\begin{equation} \label{eq:Q}
\eul(\alpha,\beta,\gamma)
\coloneqq
\zb R_3(\alpha) \zb R_2(\beta) \zb R_3(\gamma) \in\SO,
\qquad \alpha,\gamma\in\T,\ \beta\in[0,\pi],
\end{equation}
where 
\begin{equation} \label{eq:R3R2}
  \zb R_3(\alpha) \coloneqq
  \begin{pmatrix}
    \cos\alpha &-\sin\alpha &0\\
    \sin\alpha &\cos\alpha &0\\
    0&0&1
  \end{pmatrix},
  \;
  \zb R_2(\beta) \coloneqq
  \begin{pmatrix}
    \cos\beta&0 &\sin\beta \\
    0&1&0\\
    -\sin\beta&0 &\cos\beta
  \end{pmatrix}.
\end{equation}
The rotation group $\SO$ can be identified with the product $\S^2\times\T$ via the bijection 
$$
\S^2\times\T \ni (\bxi,\gamma) \mapsto \eul(\azi(\bxi), \zen(\bxi),\gamma) \in \SO,
$$
cf.\ \cite{GrPo09}.
In Euler angles, the  rotationally invariant measure $\sigma_{\SO}$ on $\SO$ is given by
\begin{align} \label{eq:SO_int}
    \int_{\SO} f(\bQ) \d\sigma_{\SO}(\bQ)
    &= 
    \int_{0}^{2\pi} \int_{0}^{\pi} \int_{0}^{2\pi}
    f(\eul(\alpha,\beta,\gamma)) \sin(\beta) \, \d\alpha \d\beta \d\gamma
    \\
    &= 
    \int_{\T} \int_{\S^2}
    f(\eul(\alpha,\beta,\gamma))
    \d\sigma_{\S^2} (\sph(\alpha,\beta)) \d\gamma.
\end{align}
The uniform measure on $\SO$ is $u_{\SO} \coloneqq (8\pi^{2})^{-1} \sigma_{\SO}$.

The \emph{rotational harmonics} or \emph{Wigner D-functions} 
$D_n^{k,j}$ of degree $n\in\NN$ and orders $k,j\in\{-n,\dots,n\}$ are
defined by 
\begin{equation} \label{eq:D}
D_n^{k,j} (\eul(\alpha,\beta,\gamma))
\coloneqq \e^{-\mathrm{i}k\alpha}\, d_n^{k,j} (\cos\beta)\, \e^{-\mathrm{i}j\gamma},
\end{equation}
where the \emph{Wigner d-functions} are given for $t\in[-1,1]$ by 
\begin{align*}
 d_n^{k,j}(t) 
&\coloneqq
\frac{(-1)^{n-j}}{2^n} 
\sqrt\frac{(n+k)!(1-t)^{j-k}}{(n-j)!(n+j)!(n-k)!(1+t)^{j+k}}
\frac{\dx^{n-k}}{\dx t^{n-k}} \frac{(1+t)^{n+j}}{(1-t)^{-n+j}},
\end{align*}
see \cite[chap.~4]{Varsha88}.
The rotational harmonics are the matrix entries of the left angular representations of $\SO$, i.e.,
\begin{equation} \label{eq:Y_D}
    Y_n^k(\bQ^\tT \bxi)
    =
    \sum_{j=-n}^{n} D_n^{j,k}(\bQ) Y_n^j(\bxi)
    ,\quad  \bQ\in\SO,\ \bxi\in\S^2.
\end{equation}
They satisfy the orthogonality relation
\begin{equation} \label{eq:D_ortho}
    \int_{\SO} D_n^{j,k}(\bQ) D_{n'}^{j',k'}(\bQ) \d \bQ
    =
    \frac{8\pi^2}{2n+1} \delta_{n,n'} \delta_{k,k'} \delta_{j,j'},
\end{equation}
for all $n,n'\in\N_0,$ $j,k=-n,\dots,n,$ and $j',k'=-n',\dots,n'$, 
where $\delta$ denotes the Kronecker symbol.
Then $\{ \big(\frac{2n+1}{8\pi^2}\big)^\frac12 D_n^{j,k}: n \in \mathbb N_0, j,k = -n,\ldots,n\}$ form an orthonormal basis
of $L^2(\SO)$.

Finally, the \emph{Sobolev space} $H^s(\SO)$ with $s\ge0$ 
is defined as the completion of $C^\infty(\SO)$ with respect to the Sobolev norm
\begin{equation}
    \norm{g}_{H^s(\SO)}^2
    \coloneqq
    \sum_{n=0}^\infty \left( n+\tfrac12\right)^{2s} \sum_{j,k=-n}^{n}
    \frac{8\pi^2}{2n+1} | \langle g, D_n^{j,k}\rangle|^2.
   \end{equation}

\section{Vertical Slice Transform}\label{sec:vert_slice}
\subsection{Vertical Slice Transform of Functions} 
In analogy to the Radon transform,
the main idea
behind the vertical slice transform
is to integrate a given function $f \colon \S^2 \to \R$
along parallel vertical slices.
To describe these slices mathematically,
we define the \emph{slicing operator} $\cS_\psi \colon \S^2 \to \II$ for any fixed $\psi \in \T$ by
\begin{equation}
  \cS_\psi (\zb \xi)
  \coloneqq
  \langle \bxi, (\cos \psi, \sin \psi,0)^\tT \rangle
  =  
  \cos(\psi) \, \xi_1 + \sin(\psi) \, \xi_2,
\end{equation}
and the corresponding \emph{slice}/\emph{circle} by
\begin{equation}
  C_{\psi}^{t} 
  \coloneqq
  \cS_\psi^{-1}(t)
  =
  \{\zb \xi \in \mathbb S^2: \cS_\psi (\zb \xi) = t\},
  \quad
  \ t \in \II.
\end{equation} 
The slice $C_{\psi}^{t}$ is the intersection of $\mathbb S^2$
and the plane with normal $(\cos \psi, \sin \psi,0)^\tT$
and distance~$t$ from the origin,
An illustration of the slices $C_\psi^t$ for fixed $\psi$ is given in \autoref{fig:vertical_slice_1}.
The \emph{vertical slice transform} $\mathcal V$ is defined by
\begin{equation} \label{eq:V}
  \mathcal V f(\psi,t)
  \coloneqq
	  \frac{1}{2 \pi \sqrt{1-t^2}} 
  \int_{C_{\psi}^{t}}
  f(\zb\xi)
  \, \ds (\bxi),
  \qquad \psi\in\T,\ t\in (-1,1),
\end{equation}
where $\mathrm{d} \text{s}$ denotes the arc-length on $C_{\psi}^{t}$.
For $t = \pm 1$,
the \emph{vertical slice transform} is
\begin{equation}
  \mathcal V f (\psi,1) \coloneqq f(\cos\psi,\sin\psi,0)
  \quad\text{and}\quad
  \mathcal V f (\psi,-1) \coloneqq f(-\cos\psi,-\sin\psi,0).
\end{equation}
For fixed $\psi \in \T$,
we define the (normalized) restrictions
\begin{equation}
  \label{eq:Vs}
  \V_\psi\coloneqq 2\pi \, \V(\psi,\cdot).
\end{equation}
This corresponds to projecting the mean values of $f$ along $C_{\psi}^{t}$ to $t \in \II$. 
For an illustration see again \autoref{fig:vertical_slice_1}.
The different normalizations of $\V$ and $\V_\psi$ are chosen
with respect to the later generalization to measures
and ensure that
density functions are transformed to density functions
by $\V$ and $\V_\psi$.
By the following proposition,
both operators are well defined almost everywhere.

\begin{proposition}
  \label{prop:V_norm}
  Let $1 \le p \le \infty$.
  For every $f \in L^p(\S^2)$,
  it holds
  \begin{equation} \label{eq:V_int}
    \int_{\II} \V_\psi f(t) \, \dx t
    =
    \int_{\S^2} f(\zb\xi) \, \dx\sigma_{\S^2}(\zb\xi)
    \quad\text{and}\quad
    \int_{\T}\int_{\II} \V f(\psi, t) \, \dx t \d \psi
    =
    \int_{\S^2} f(\zb\xi) \, \dx\sigma_{\S^2}(\zb\xi).
  \end{equation}
  Let $\psi \in \T$.
  The operators
  $\V_\psi \colon L^p(\S^2) \to L^p(\II)$
  and
  $\V \colon L^p(\S^2) \to L^p(\T \times \II)$
  are bounded  with
  \begin{equation}
    \lVert \V_\psi \rVert_{L^p \to L^p} = (2\pi)^{1-1/p} 
    \quad\text{and}\quad
    \lVert \V \rVert_{L^p \to L^p} = 1.
  \end{equation}
  Moreover,
  it holds
  $\V_\psi \colon C(\S^2) \to C(\II)$
  and
  $\V \colon C(\S^2) \to C(\T \times \II)$.
\end{proposition}

\begin{proof}
  We parameterize the \emph{upper} and \emph{lower hemispheres} by
  \begin{equation}
    H_{\psi}^\pm (s,t)
    \coloneqq
    \begin{pmatrix}
      t \cos(\psi) - s \sin(\psi)\\
      t \sin(\psi) + s \cos(\psi)\\
      \pm \sqrt{1- t^2 - s^2}
    \end{pmatrix}, \quad s \in \sqrt{1-t^2} \; \II, \ t \in \II.
  \end{equation}
  Then the upper and lower semicircle of 
  $C_{\psi}^{t}$ can be parameterized via $H_{\psi}^\pm(\cdot,t)$.
  Thus we obtain
  \begin{align}
    \int_{\mathbb S^2} f(\zb \xi) \, \dx\sigma_{\S^2}(\zb \xi)
    &= 
      \int_{\II} \int_{\sqrt{1-t^2}\,\II}
      \left( f(H_{\psi}^+(s,t) + f(H_{\psi}^-(s,t) )\right) 
      \frac{1}{\sqrt{1-t^2-s^2} } \, \dx s \, \dx t
    \\
    &=
      \int_{\II} 
      \frac{1}{\sqrt{1-t^2}}
      \int_{C_{\psi}^{t}}
      f(\zb \xi) \,
      \ds(\bxi) \, \dx t
    =
    \int_{\II} \V_\psi f(t) \, \dx t.
  \end{align}
  Using \eqref{eq:Vs} and integrating over $\psi$ 
  immediately yields the second identity in \eqref{eq:V_int}.
  By Fubini's theorem, $\V_\psi$ and $\V$ are well defined. 
  
  Following the above computation for the absolute value of $f$, we obtain with the triangle inequality
  $\lVert \V_\psi \rVert_{L^1 \to L^1} = \lVert \V \rVert_{L^1 \to L^1} = 1$.
  Since the vertical slice transform is essentially bounded by
  \begin{equation}
    \label{eq:6}
    \lvert \V_\psi f(t) \rvert
    \le
    \frac{1}{\sqrt{1-t^2}}
    \int_{C_{\psi}^{t}} \lvert f( \bxi ) \rvert \ds(\bxi)
    \le
    2\pi \esssup_{\bxi \in \S^2} \, \lvert f(\bxi) \rvert,
  \end{equation}
  we further have $\lVert \V_\psi \rVert_{L^\infty \to L^\infty} = 2\pi$
  and $\lVert \V \rVert_{L^\infty \to L^\infty} = 1$.
  Now the second assertion follows from the Riesz--Thorin interpolation theorem.
  
  The last assertion is an immediate consequence of Lebesgue's dominated convergence theorem.
\end{proof}

Since all circles $C_{\psi}^{t}$ are symmetric with respect to the $\xi_1$-$\xi_2$ plane,
$\mathcal V f$ vanishes for functions~$f$ which are odd in the third coordinate,
i.e., $f(\xi_1,\xi_2,\xi_3) = - f(\xi_1,\xi_2,-\xi_3)$.
For brevity,
we call these functions \emph{odd}.
In \cite{GiReSh94},
an explicit inversion formula for \emph{even} functions, i.e., $f(\xi_1,\xi_2,\xi_3) = f(\xi_1,\xi_2,-\xi_3)$, is derived. 
However,
as for the Radon inversion formula,
this formula leads to instable practical computations
if we leave the range of $\V$.
For numerical simulation,
we will invert $\V$ using its singular value decomposition.
For this purpose,
notice that
the spherical harmonics $Y^k_n$ with even $k+n$ are even functions, while
those with odd $k+n$ are odd functions.

\begin{theorem}[\!{\cite[Thm.~3.3]{HiQu15circav}}] \label{prop:V-inj}
  The vertical slice transform \eqref{eq:V} fulfills
  \begin{equation} \label{eq:VY}
    \V Y_n^k (\psi,t)
    =
    \mathrm{v}_n^k\, \sqrt{\tfrac{2n+1}{4 \pi}} \, \e^{\i k \psi}\, P_n(t)
    ,\qquad n\in\NN,\ k\in\{-n,\dots,n\},\ n+k\text{ even},
  \end{equation}
  where
  \begin{equation}
    \mathrm{v}_n^k \coloneqq (-1)^{\frac{n+k}2} \sqrt{\frac{(n-k)!}{(n+k)!}} \frac{(n+k-1)!!} {(n-k)!!} .
  \end{equation}
  There exist constants $C_1,C_2>0$ such that for all $n\in\NN$, $k\in\{-n,\dots,n\}$ with $n+k$ even,
  \begin{equation} \label{eq:V-sv-bound}
    C_1(n+1/2)^{-1/2}
    \le
    \lvert \mathrm{v}_n^k\rvert
    \le
    C_2(n+1/2)^{-1/4}.
  \end{equation}
\end{theorem}

Noting that the functions
\begin{equation} \label{eq:Bnk}
  B_n^k(\psi,t)
  \coloneqq
  \sqrt{\frac{2n+1}{4\pi}}\, P_n(t) \, \e^{\i k \psi}
  ,\qquad \forall (\psi,t) \in \T\times \II,
\end{equation}
form an orthonormal basis of $L^2(\T \times \II)$ and that
$\mathrm{v}_n^k \to 0$ as $n \to \infty$,
we deduce that
$\V\colon L^2(\S^2) \to L^2(\T \times \II)$ is a compact operator 
with singular value decomposition
\begin{equation} \label{eq:V_svd}
  \V f (\psi,t) = \sum_{n \in {\mathbb N_0}}
  \sum_{\substack{k=-n\\n+k \,  \text{even}} }^n \mathrm{v}_n^k \, B_n^k(\psi,t) 
  .
\end{equation}
Restricting $\V$ to even functions $L^2_{\sym}(\S^2)$,
where
$L^p_{\sym}(\S^2)$ with $1 \le p \le \infty$ is defined as
\begin{equation}
  L^p_{\sym}(\S^2)
  \coloneqq
  \bigl\{f\in L^p(\S^2) : f(\xi)= \check f(\xi)
  \text{ a.e. on }\S^2 \bigr\}
\end{equation}
and $\check f(\bxi) \coloneqq f(\xi_1,\xi_2,-\xi_3)$ is
the reflection at the $\xi_1$-$\xi_2$ plane,
the operator $\V\colon L^2_{\sym}(\S^2) \to L^2(\T \times \II)$ is injective.
Its \emph{Moore--Penrose pseudoinverse}, cf.\ \cite{EnHaNe96}, is given by
\begin{equation}\label{MP_V}
  \V^\dagger \colon \mathcal R(\V)\oplus \mathcal R(\V)^\perp \to L^2_{\sym}(\S^2)
  ,\qquad
  \V^\dagger g 
  = \sum_{n \in {\mathbb N_0}} 
  \sum_{\substack{k=-n\\n+k \,  \text{even}} }^n
  (\mathrm{v}_n^k)^{-1} \,
  \langle g, B_n^k \rangle \,
  Y^k_n,
\end{equation}
where $\mathcal R(\V)$ denotes the range of $\V$.
We will further need the adjoint operator of $\mathcal V$.

\begin{proposition}
  Let $1 \le p,q \le \infty$ with $1/p + 1/q = 1$.
  For $1 \le p < \infty$,
  the adjoint $\V^*\colon L^q(\T \times \II) \to L^q(\S^2)$
  of $\V\colon L^p(\S^2) \to L^p(\T \times \II)$ is given by
  \begin{equation}
    \label{eq:V*}
    \mathcal V^*g(\bxi)
    = 
    \frac{1}{2\pi}
    \int_{\T} g(\psi,\xi_1\cos\psi+\xi_2\sin\psi) \d\psi,
  \end{equation}
  and the adjoint $\V_\psi^*\colon L^q(\II) \to L^q(\S^2)$
  of $\V_\psi \colon L^p(\S^2) \to L^p(\II)$ by
  \begin{equation}
    \label{eq:Vs*}
    \mathcal V_\psi^*g(\bxi)
    = 
    g(\xi_1\cos\psi+\xi_2\sin\psi).
  \end{equation}
  Moreover,
  it holds $\V^* \colon C(\T \times \II) \to C(\S^2)$
  and $\V_\psi^* \colon C(\II) \to C(\S^2)$.
\end{proposition}

\begin{proof}
  The assertion follows from \autoref{prop:V_norm},
  which yields
  \begin{align}
    \langle \V f,g \rangle 
    &=
      \int_{\T} \int_{\II}
      \V f (\psi,t) \, g(\psi,t) \,
      \d t \,\d\psi
      =
      \int_{\T} \int_{\II}
      \frac{1}{2\pi\sqrt{1-t^2}}
      \int_{C_{\psi}^{t}} f(\bxi)
      \, g(\psi,t) \, \ds(\bxi) \d t \d\psi
    \\
    &=
      \frac{1}{2\pi}
      \int_{\T} \int_{\S^2}
      f(\bxi) \, g(\psi,\xi_1\cos\psi+\xi_2\sin\psi)
      \d\sigma_{\S^2}(\bxi) \d\psi
      =
      \langle f,  \V^*g  \rangle
  \end{align}
  for all $f \in L^p(\S^2)$, $g \in L^q(\T \times \II)$.
  The adjoint of $\V_\psi$ can be established analogously%
  ---without the integral over $\T$ and the factor $(2\pi)^{-1}$.
  The last assertion again follows from
  Lebesgue's dominated convergence theorem
  and by the definition of the adjoint.
\end{proof}

\subsection{Vertical Slice Transform of Measures}
\label{sec:vert-slice-meas}
For functions $f\colon \S^2\to\R$, the vertical slice transform $\V f(\psi,t)$ in \eqref{eq:V}
and its restriction $\V_\psi f(t)$ in \eqref{eq:Vs}
are integrals of $f$ along the slices $\cS_\psi^{-1}(t)$.
Heuristically,
the related concept for measures $\mu \in \M(\S^2)$
would be to consider $\mu(\cS_\psi^{-1}(t))$.
In this manner,
for a fixed angle $\psi \in \T$,
we generalize the (restricted) \emph{vertical slice transform} $\V_\psi$ by
\begin{equation}
  \label{eq:Vs_measure}
  \V_\psi \colon \M(\S^2) \to \M(\II),\quad
  \mu \mapsto (\cS_\psi)_\# \mu = \mu \circ \cS_\psi^{-1}.
\end{equation}
In the function setting,
we figuratively obtain $\V f$ by gluing the (rescaled) functions $\frac{1}{2\pi}\V_\psi f$
together along the angle $\psi$.
In the measure setting,
the corresponding concept is to consider $\V_\psi$
as disintegration family.
We define the \emph{vertical slice transform} $\V\colon \M(\S^2) \to \M(\T \times \II)$ by
\begin{equation}
  \label{def:V_measure}
  \V \mu
  \coloneqq
  (T_{\V})_\# (u_{\T} \times \mu)
  \quad\text{with}\quad
  T_{\V}(\psi, \zb\xi) \coloneqq (\psi, \cS_\psi(\zb\xi)).
\end{equation}
The disintegration aspect becomes clear in the following proposition.

\begin{proposition}
  \label{prop:V-dis}
  Let $\mu \in \M(\S^2)$.
  Then $\V \mu$ can be disintegrated into the family $\V_\psi \mu$
  with respect to the uniform measure $u_\T$,
  i.e., for all $g \in C(\T\times\II)$, it holds
  \begin{equation}
    \label{eq:13}
    \int_{\T \times \II}
    g(\psi, t) 
    \d \V\mu (\psi, t)
    =
    \int_{\T} \int_{\II}
    g(\psi, t)
    \d \V_\psi \mu(t) \d u_\T(\psi).
  \end{equation}
\end{proposition}

\begin{proof}
  Incorporating \eqref{def:V_measure},
  and using Fubini's theorem,
  we obtain
  \begin{equation}
    \langle \V\mu, g \rangle
    =
    \int_{\T} \int_{\S^2}
    g(\psi, \cS_\psi (\zb \xi))
    \d \mu(\zb \xi) \d u_\T(\psi)
    =
    \int_{\T} \int_{\II}
    g(\psi, t)
    \d ( (\cS_\psi)_\# \mu )(t) \d u_\T(\psi)
  \end{equation}
  for every $g  \in C(\T \times \II)$.
  By \eqref{eq:Vs_measure} this implies the assertion.
\end{proof}

The defined measure-valued versions of $\V$ and $\V_\psi$
are in fact the adjoints of
$\V^* \colon C(\T \times \II) \to C(\S^2)$ in \eqref{eq:V*}
and
$\V_\psi^* \colon C(\II) \to C(\S^2)$ in \eqref{eq:Vs*},
which explains the generalizations from the duality point of view.

\begin{proposition}
  \label{prop:V-as-adj}
  The vertical slice transforms
  \eqref{def:V_measure} and \eqref{eq:Vs_measure}
  satisfy
  \begin{align}
    \label{eq:V_measure_adj}
    \langle \V\mu, g \rangle
    &=
      \langle \mu, \V^*g \rangle
      \quad \text{for all } g \in C(\T \times \II) \quad \text{and}
    \\
    \langle \V_\psi \mu, g \rangle
    &=
      \langle \mu, \V_\psi^*g \rangle
      \quad \text{for all } g \in C(\II),\, \psi\in\T
  \end{align}
  with the adjoint operators from \eqref{eq:V*} and \eqref{eq:Vs*}.
\end{proposition}

\begin{proof}
  For $\mu \in \M(\S^2)$ and $g \in C(\T \times \II)$,
  the conjecture can be established by
  \begin{equation} 
    \langle \V\mu, g \rangle
    =
    \int_{\T \times \II} 
    g(\psi,t) \, 
    \d (T_{\V})_\# (u_{\T} \times \mu)(\psi,t)
    =
    \int_{\S^2} \int_{\T} 
    g(\psi, \cS_\psi(\zb \xi)) 
    \d u_{\T}(\psi) \d \mu(\zb \xi)
    = 
    \langle \mu, \V^*g \rangle
  \end{equation}
  and,
  for $\mu \in \M(\S^2)$, $g \in C(\II)$, and fixed $\psi \in \T$,
  by
  \begin{equation*} 
    \langle \V_\psi\mu, g \rangle
    =
    \int_{\II} 
    g(t) \, 
    \d (\cS_\psi)_\#  \mu (t)
    =
    \int_{\S^2}  
    g(\cS_\psi(\zb \xi)) 
    \d \mu(\zb \xi)
    = 
    \langle \mu, \V_\psi^*g \rangle. \qedhere
  \end{equation*}
\end{proof}

One could equivalently use the identity \eqref{eq:V_measure_adj} to define the vertical slice transform of a measure, analogously as it was done for the Radon transform in \cite[Chap.\ 2, § 2]{Hel11}.
For absolutely continuous measures with respect to $\sigma_{\S^2}$,
the measure- and function-valued vertical slice transforms coincide,
which now justify the different scalings in \eqref{eq:V} and \eqref{eq:Vs}.

\begin{proposition}
  \label{cor:abs1}
  For $f \in L^1(\S^2)$,
  the vertical slice transforms satisfy
  \begin{equation*}
    \V[f \sigma_{\S^2}] = (\V f) \, \sigma_{\T \times \II}
    \quad\text{and}\quad
    \V_\psi[f \sigma_{\S^2}] = (\V_\psi f) \, \sigma_{\II}.
  \end{equation*}
  In particular,
  the transformed measures are again absolutely continuous.
\end{proposition}

\begin{proof}
Let $\langle \cdot, \cdot \rangle_\M$
  denotes the dual pairing for measures and continuous function
  and $\langle \cdot, \cdot \rangle_{L}$
  the dual pairing between $L^1$ and $L^\infty$ functions.
  Then the identity follows directly from \autoref{prop:V-as-adj} by
  \begin{equation}
    \label{eq:2}
    \langle \V[f \sigma_{\S^2}], g \rangle_{\M}
    =
    \langle f \sigma_{\S^2}, \V^* g \rangle_{\M}
    =
    \langle f, \V^* g \rangle_{L}
    =
    \langle \V f, g \rangle_{L}
    =
    \langle (\V f)\, \sigma_{\T \times \II}, g \rangle_{\M}
  \end{equation}
  for all $g \in C(\T\times\II)$.
  For $\V_\psi$, the identity follows analogously.
\end{proof}

By the following theorem, we see that similarly to the function setting,
the vertical slice transform is injective
when restricted to the even measures
(with respect to the $\xi_1$-$\xi_2$ plane)
given by
\begin{equation}
  \M_{\sym}(\S^2)
  \coloneqq
  \{\mu\in\M(\S^2) : \inn{\mu,f} = \inn{\mu,\check f} \text{ for all } f\in C(\S^2) \}.
  \label{eq:Msym}
\end{equation}

\begin{theorem} \label{thm:V_inj} 
  The vertical slice transform
  $\V \colon \M_\mathrm{sym}(\S^2)\to \M(\T\times \II)$ 
  is injective. 
\end{theorem}

The proof is given in \autoref{app:thm:V_inj}.

\section{Normalized Semicircle Transform}\label{sec:semicircle}
\begin{figure}
  \centering
    \includegraphics[width=.4\textwidth]{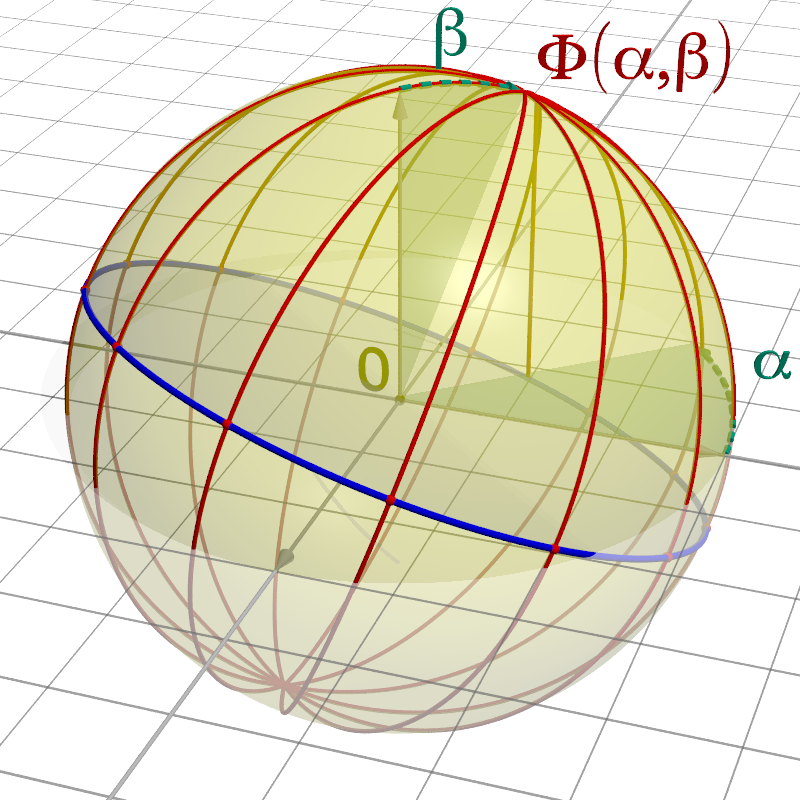}
  \caption{
    Semicircles $M_{\alpha,\beta}^\gamma$ (red) 
    starting at a fixed point $\Phi(\alpha,\beta)$ and with varying $\gamma \in\T$.
    Here $\beta$ is the angle of $\Phi(\alpha,\beta)$ to the north pole and $\alpha$ the angle of its projection in the $\xi_1$-$\xi_2$ plane to the $\xi_1$ axis.
    The blue circle is orthogonal to the semicircles.}	
  \label{fig:semicircles} 
\end{figure}	

\subsection{Normalized Semicircle Transform of Functions} 
Instead of integrating over parallel slices,
the semicircle transform
integrates a function along all meridians
with respect to a fixed zenith on the sphere.
For any \emph{zenith} $\sph(\alpha,\beta)\in\S^2$ with $\alpha \in \T$ and $\beta \in [0,\pi]$,
we define the \emph{azimuth operator} $\cA_{\alpha,\beta}\colon \S^2 \to \T$
and the \emph{zenith operator} $\cZ_{\alpha,\beta}\colon \S^2 \to [0, \pi]$ as
\begin{align}
   \label{eq:4}
  \cA_{\alpha,\beta}(\bxi)
  &\coloneqq
  \azi(\eul(\alpha,\beta,0)^\tT \, \bxi)
  ,\\
  \cZ_{\alpha,\beta}(\bxi)
  &\coloneqq
  \zen(\eul(\alpha,\beta,0)^\tT \, \bxi),
\end{align}
i.e.,
we rotate the zenith back to the north pole
and take the azimuth and zenith angle,
see \autoref{fig:semicircles}.
For the zenith $\sph(\alpha,\beta)$
and fixed $\gamma \in \T$,
we consider the \emph{semicircles/meridians}
\begin{equation}
  M_{\alpha,\beta}^\gamma
  \coloneqq
  \cA_{\alpha,\beta}^{-1}(\gamma)
  =
  \{\bxi \in \mathbb S^2 : \cA_{\alpha,\beta}(\bxi) = \gamma \}.
\end{equation}%
If $\gamma\neq0$, we have
\begin{equation}
 M_{\alpha,\beta}^\gamma
 = 
 \{
\bxi=\eul(\alpha,\beta,0) \sph(\gamma, \vartheta) 
= 
\eul(\alpha,\beta,\gamma) \sph(0,\vartheta): \vartheta \in (0,\pi) 
\}.
\end{equation}
Otherwise, if $\gamma=0$, we need to replace the open interval by a closed one, i.e., $\vartheta \in [0,\pi]$.
Figuratively,
$M_{\alpha,\beta}^\gamma$ is a rotation of
the meridian $\{\sph(\gamma, \vartheta) : \vartheta \in (0, \pi)\}$
with azimuth $\gamma$ by $\eul(\alpha,\beta,0)$.
The \emph{normalized semicircle transform} $\mathcal W$
of $f \colon \S^2 \to \R$ is defined by
\begin{align}
  \label{eq:7}
  \mathcal W f (\alpha,\beta,\gamma)
  &\coloneqq
  \frac{1}{4\pi}
  \int_{M_{\alpha,\beta}^\gamma}
  f(\bxi) \, \sin\left(\cZ_{\alpha,\beta} ( \bxi) \right)\,
  \ds(\bxi)\\
  &=
  \frac{1}{4\pi}
  \int_0^\pi
  f(\eul(\alpha,\beta,0) \sph(\gamma,\vartheta))
  \sin(\vartheta)
  \dx \vartheta.
\end{align}
We may interpret $M_{\alpha,\beta}^\gamma$ as rotation of the prime median by $\eul(\alpha,\beta,\gamma)$.
Based on the substitution $\bQ = \eul(\alpha,\beta,\gamma)$, the normalized semicircle transform defines a function on $\SO$ via
\begin{equation}
  \label{eq:B}
  \mathcal W f (\bQ)
  \coloneqq
  \frac{1}{4\pi}
  \int_{0}^{\pi}
  f\left(\bQ \sph(0,\vartheta) \right) \, \sin(\vartheta) \,
  \dx \vartheta.
\end{equation}
Henceforth,
we will not distinguish between
$\W f(\alpha, \beta, \gamma)$ and $\W f(\bQ)$.
Especially
for the inversion formula by the singular value decomposition,
we will make use of the latter definition.
The multiplication with $(4\pi)^{-1}\, \sin (\vartheta)$ in the latitude
ensures that
density functions are mapped to density functions
allowing the later generalization to measures.
For the zenith $\sph(\alpha,\beta)$,
we define the (normalized) restriction
\begin{equation}
  \label{eq:5}
  \W_{\alpha,\beta} f
  \coloneqq
  4\pi \, \W f(\alpha,\beta,\cdot).
\end{equation}

\begin{remark}
  The \emph{(unnormalized) semicircle transform} $\widetilde {\mathcal W}\colon C(\S^2)\to C(\SO)$ is defined 
  by
  \begin{equation} \label{eq:A}
    \widetilde {\mathcal W} f(\bQ) \coloneqq \int_{-\pi/2}^{\pi/2} f\left(\smash{\bQ^\tT (\sph(\varphi,\tfrac\pi2))}\right) \d\varphi
    ,\qquad \bQ\in\SO,
  \end{equation}
  see \cite{HiPoQu18}.
  It computes the mean values of $f$ along all half great circles of the sphere, i.e., without the weight $\sin(\vartheta)$ of \eqref{eq:B}.
  The injectivity of $\widetilde {\mathcal W}$ was shown in \cite{Gro98}.
  A singular value decomposition and inversion algorithms were provided in \cite{HiPoQu18}.
  The authors of \cite{Bon22} reinvented this transform with another parameterization using the plane through $\Psi(\alpha,\beta,0) \zb e^1$ and $\Psi(\alpha,\beta,0) \zb e^2$.
More precisely, their notation was not clear to us since it seems that they have applied the normalized transform
in the numerical examples, but certain parts in their analysis rely on the unnormalized transform.
\end{remark}

The semicircle transforms $\W$ and $\W_{\alpha,\beta}$
are well defined for continuous functions
as well as
for $p$-integrable functions.
Moreover,
both transforms are continuous operators.

\begin{proposition}
  \label{prop:important}
  Let $1 \le p \le \infty$,
  and let $\sph(\alpha,\beta) \in \S^2$.
  For every $f \in L^p(\S^2)$,
  it holds 
  \begin{equation}
    \int_{\T} \W_{\alpha,\beta} f(\gamma) \, \dx \gamma
    =
    \int_{\S^2} f(\zb\xi) \, \dx\sigma_{\S^2}(\zb\xi)
    \quad\text{and}\quad
    \int_{\SO} \W f(\bQ) \d\sigma_{\SO}(\bQ)
    =
    \int_{\S^2} f(\zb\xi) \, \dx\sigma_{\S^2}(\zb\xi).
  \end{equation}
  The operators
  $\W_{\alpha,\beta} \colon L^p(\S^2) \to L^p(\T)$
  and
  $\W \colon L^p(\S^2) \to L^p(\SO)$
  are bounded  with
  \begin{equation}
    \lVert \W_{\alpha,\beta} \rVert_{L^p \to L^p} \le 2^{1-1/p} 
    \quad\text{and}\quad
    \lVert \W \rVert_{L^p \to L^p} \le (2\pi)^{1/p - 1}.
  \end{equation}
  Moreover,
  it holds
  $\W_{\alpha,\beta} \colon C(\S^2) \to C(\T)$
  and
  $\W \colon C(\S^2) \to C(\SO)$.
\end{proposition}

\begin{proof}
  Since the surface measure on $\S^2$ is invariant under rotations, we have
  \begin{align}
    \int_{\T} \mathcal W_{\alpha,\beta} f(\gamma) \, \dx\gamma
    &=
      \int_{\T} \int_{0}^{\pi}
      f ( \eul(\alpha,\beta,0)\, \sph(\gamma,\vartheta) ) \sin(\vartheta)
      \d \vartheta \d\gamma
    \\
    &=
      \int_{\S^2} f ( \eul(\alpha,\beta,0) \, \bxi ) \d\sigma_{\S^2} (\bxi)
      =
      \int_{\S^2} f (\bxi) \d\sigma_{\S^2}(\bxi). 
  \end{align}
  The definition of $\W_{\alpha,\beta}$
  and integration over $\alpha$ and $\beta$
  gives the second identity.
  Thus $\W_{\alpha,\beta}$ and $\W$ are well defined almost everywhere
  by Fubini's theorem.
  Using absolute values and the triangle inequality
  in the above computation
  yields $\lVert \W_{\alpha,\beta} \rVert_{L^1 \to L^1} = \lVert \W \rVert_{L^1 \to L^1} = 1$,
  where we deduce a lower bound of the norm by inserting the constant $f=1$.
  The semicircle transform is further essentially bounded by
  \begin{equation}
    \label{eq:8}
    \lvert \W_{\alpha,\beta} f(\gamma) \rvert
    \le
    \int_{0}^\pi \lvert 
    f\left( \eul(\alpha,\beta,0) \sph(\gamma,\vartheta ) \right)\rvert \sin(\vartheta) \dx\vartheta
    \le
    2 \esssup_{\bxi \in \S^2} \, \lvert f(\bxi) \rvert;
  \end{equation}
  so $\lVert \W_{\alpha,\beta} \rVert_{L^\infty \to L^\infty} = 2$
  and $\lVert \W \rVert_{L^\infty \to L^\infty} = (2\pi)^{-1}$.
  The second assertion now follows from the Riesz--Thorin interpolation theorem.
  The last assertion is an immediate consequence of Lebesgue's dominated convergence theorem.
\end{proof}

Considering the semicircle transform in the Hilbert space setting,
i.e.,
$\W \colon L^2(\S^2) \to L^2(\SO)$,
we are interested in its singular value decomposition.

\begin{theorem} \label{thm:svd2}
  The normalized semicircle transform fulfills 
  \begin{equation} \label{eq:B_svd}
    \mathcal W Y_n^k = \mathrm{w}_n Z_n^k
    ,\quad  n\in\NN,\ k\in\{-n,\dots,n\},
  \end{equation}
  with the singular values
  $\mathrm{w}_n \coloneqq
  \lVert\mathcal W Y_n^k\rVert_{L^2(\S^2)}
  $
  and the orthonormal functions
  \begin{equation} \label{eq:Z}
    Z_n^k \coloneqq \mathrm{w}_n^{-1} \sum_{j=-n}^n \lambda_n^j\, \overline{D_n^{k,j}} \in L^2(\SO),
  \end{equation}
  where $\lambda_0^0 \coloneqq 2 (4\pi)^{-3/2}$ and, for $n\in\N$ and $j\in\{1,\dots,n\}$ with $n+j$ even,
  \begin{equation} \label{eq:B_sv}
    \lambda_n^j
    \coloneqq
    \frac{(-1)^j}{4\pi} \sqrt{\frac{2n+1}{4\pi} \frac{(n-j)!}{(n+j)!}}\, \frac{j\, (n-2)!!\, (n+j-1)!!}{ (n-j)!!\, (n+1)!!} 
    \begin{cases}
      2: & n \text{ even},\\
      {\pi}: & n \text{ odd},
    \end{cases}
  \end{equation}
  $\lambda_n^{-j} \coloneqq (-1)^j \lambda_n^j$, and $\lambda_n^j=0$ otherwise.
  Here $Y_n^k$ denote the spherical harmonics \eqref{eq:Y} and $D_n^{k,j}$ the rotational harmonics \eqref{eq:D}.
  Moreover,
  there are constants $C_1,C_2>0$ such that
  \begin{equation} \label{eq:mu-bound}
    C_1\ (n+1)^{-1/2}
    \le
    \mathrm{w}_n
    \le
    C_2\ (n+1)^{-1/2}
    \quad \text{for all } n\in\NN.
  \end{equation}
\end{theorem}

The proof is given in \autoref{sec:W-svd}.
Analogously to \cite[Thm.\ 3.13]{quellmalzdiss}, we see that
the semicircle transform is a smoothing operator.

\begin{corollary}
For  $s\ge0$, 
the operator $\mathcal W\colon H^s(\S^2) \to H^{s+1/2}(\SO)$ is continuous.
\end{corollary}

Theorem \ref{thm:svd2}
implies that
$\mathcal W\colon L^2(\S^2) \to L^2(\SO)$ is an injective, compact operator with
\begin{equation}
  \mathcal W f  = \sum_{n \in \NN} \mathrm{w}_n
  \sum_{k=-n}^n \langle f, Y^k_n \rangle Z^k_n.
  \label{eq:9}
\end{equation}
Its \emph{Moore--Penrose pseudoinverse} is given by
\begin{align}
  \mathcal W^\dagger g
  &=
    \sum_{n=0}^{\infty} \frac{1}{\mathrm{w}_n} \sum_{k=-n}^{n}
    \inn{g, Z_n^{k}}
    Y_n^k
  =
  \sum_{n=0}^{\infty} \frac{1}{(\mathrm{w}_n)^2} \sum_{k=-n}^{n}
  \sum_{\substack{{j}=-n 
  }}^n
  {\lambda_n^j}
  \inn{g, \overline{D_n^{k,j}}}
  Y_n^k
  .    \label{eq:Bdag}
\end{align}

We will also need the adjoint operator.

\begin{proposition}
  Let $1 \le p,q \le \infty$ with $1/p + 1/q = 1$.
  For $1 \le p < \infty$,
  the adjoint $\W^*\colon L^q(\SO) \to L^q(\S^2)$
  of $\W\colon L^p(\S^2) \to L^p(\SO)$ is given by
  \begin{equation}
    \label{eq:B*}
    \mathcal W^* g (\bxi)
    =
    \frac1{4\pi} \int_{\T} \int_{0}^{\pi}
    g (\eul(\alpha, \beta, \cA_{\alpha,\beta} \, \bxi))
    \,\sin(\beta) \d\beta \d\alpha,
  \end{equation}
  and the adjoint $\W_{\alpha,\beta}^*\colon L^q(\T) \to L^q(\S^2)$
  of $\W_{\alpha,\beta} \colon L^p(\S^2) \to L^p(\T)$ by
  \begin{equation}
    \label{eq:Bab*}
    \W_{\alpha,\beta}^*g(\bxi)
    = 
    g( \cA_{\alpha,\beta} \, \bxi).
  \end{equation}
  Moreover,
  it holds $\W^* \colon C(\SO) \to C(\S^2)$,
  but $\W_{\alpha,\beta}^* \colon C(\T) \not\to C(\S^2)$.
\end{proposition}

\begin{proof}
  Let $f\in L^p(\S^2)$ and $g\in L^q(\SO)$.
  Based on \eqref{eq:SO_int} and \eqref{eq:B},
  and using the substitution $\zb\eta \coloneqq \sph(\gamma,\vartheta)$
  with $\gamma = \azi(\zb\eta)$,
  we compute the adjoint by 
  \begin{align}
    \langle \W f, g\rangle
    &=
      \frac1{4\pi}
      \int_{\T} \int_{0}^{\pi} \int_{\T} \int_0^\pi
      f(\eul(\alpha,\beta,0) \sph(\gamma,\vartheta))\,
      g(\eul(\alpha,\beta,\gamma))\,
      \sin(\vartheta) \sin(\beta) \d\vartheta \d \alpha  \d\beta \d\gamma
    \\
    &=
      \frac1{4\pi}
      \int_{\T} \int_{0}^{\pi} \int_{\S^2} f(\eul(\alpha,\beta,0)\, \zb\eta)\,
      g(\eul(\alpha,\beta,\azi(\zb\eta)) \sin(\beta) \d\sigma_{\S^2}(\zb\eta)  \d\beta  \d\alpha 
    \\
    &=
      \frac1{4\pi}
      \int_{\S^2} \int_{\T} \int_{0}^{\pi}  f(\zb\xi)\,
      g(\eul(\alpha,\beta,\cA_{\alpha,\beta}(\zb\xi))
      \sin(\beta)  \d\beta \d\alpha \d\sigma_{\S^2}(\zb\xi)
      =
      \langle f, \W^* g \rangle,
  \end{align}
  where we used  $\zb\xi \coloneqq \eul(\alpha,\beta,0)\, \zb\eta$ in the last line.
  The adjoint of $\W_{\alpha,\beta}$ follows analogously.
  The continuity of $\W^* g$ for $g \in C(\SO)$ can be established by
  Lebesgue's dominated convergence theorem.
  For non-constant $g\in C(\T)$,
  the adjoint $\W_{\alpha,\beta}^*$ is discontinuous at $\sph(\alpha,\beta)$.
\end{proof}

\subsection{Normalized Semicircle Transform of Measures}
The generalization from functions to measures can be done
analogously to \autoref{sec:vert-slice-meas}.
For the zenith $\sph(\alpha,\beta)$,
we generalize the (restricted) \emph{semicircle transform} $\W_{\alpha,\beta}$ by
\begin{equation}
  \label{eq:Wab-meas}
  \W_{\alpha,\beta} \colon \M(\S^2) \to \M(\T)
  ,\quad
  \mu \mapsto (\cA_{\alpha,\beta})_\# \mu
  = \mu \circ \cA_{\alpha,\beta}^{-1}.
\end{equation}
Considering the measures $\W_{\alpha,\beta} \, \mu$ as disintegration family,
we define the \emph{(normalized) semicircle transform}
$\W \colon \M(\S^2) \to \M(\SO)$ by
\begin{equation} \label{def:sctm}
  \mathcal W \mu
  \coloneqq
  (T_{\mathcal W})_\# (u_{\S^2} \times \mu)
  \quad\text{with}\quad
  T_{\mathcal W}(\sph(\alpha,\beta), \zb \xi)
  \coloneqq
  \eul(\alpha,\beta, \cA_{\alpha,\beta}(\zb\xi)).
\end{equation}

\begin{proposition}
  Let $\mu \in \M(\S^2)$.
  Then $\W \mu$ can be disintegrated
  into the family $\W_{\alpha,\beta} \, \mu$ with respect to the uniform measure $u_{\S^2}$,
  i.e.,
  for all $g \in C(\SO)$,
  it holds
  \begin{equation}
    \int_{\SO}
    g(\bQ) \d(\W\mu)(\bQ)
    =
    \int_{\S^2} \int_\T
    g(\eul(\alpha,\beta,\gamma))
    \d (\W_{\alpha,\beta} \mu)(\gamma)
    \d \sigma_{\S^2}(\sph(\alpha,\beta)).
  \end{equation}
\end{proposition}

\begin{proof}
  Inserting \eqref{def:sctm}
  and using Fubini's theorem,
  we obtain
  \begin{align}
    \langle \W \mu, g \rangle
    &=
      \int_{\S^2} \int_{\S^2}
      g(\eul(\alpha,\beta,\cA_{\alpha,\beta} \, \bxi))
      \d \mu(\zb \xi) \d u_{\mathbb S^2} (\sph(\alpha,\beta))\\
    &=
      \int_{\S^2} \int_{\T}
      g(\eul(\alpha, \beta, \gamma))
      \d (\cA_{\alpha,\beta})_\# \mu(\gamma) \d u_{\mathbb S^2}(\sph(\alpha,\beta)).
  \end{align}
  for every $g\in C(\SO)$
  establishing the asserstion.
\end{proof}

While $\W$ can be interpreted as the adjoint of $\W^*\colon C(\SO) \to C(\S^2)$ in \eqref{eq:B*},
the same reasoning does not hold for $\W_{\alpha,\beta}$ by the lack of continuity of $\W_{\alpha,\beta}^*g$ for continuous $g$.

\begin{proposition}
  \label{prop:W-adj-meas}
  The semicircle transforms \eqref{def:sctm} and \eqref{eq:Wab-meas} satisfy 
  \begin{align}
    \langle \W \mu, g \rangle
    &=
      \langle \mu, \W^* g \rangle
      \quad
      \text{for all } g \in C(\SO) \text{ and}
    \\
    \langle \W_{\alpha,\beta} \mu, g \rangle
    &=
      \int_{\S^2} \W_{\alpha,\beta}^* \, g (\bxi) \d \mu(\bxi)
      \quad
      \text{for all } g \in C(\T)
      \label{eq:Wab-meas-adj}
  \end{align}
  with the adjoint operator from \eqref{eq:B*} and \eqref{eq:Bab*}.
\end{proposition}

\begin{proof}
  Plugging in the push-forward definition \eqref{def:sctm}, we obtain
  \begin{equation}
    \langle \mathcal W \mu, g \rangle
    = 
    \int_{\S^2} \int_{\S^2}
    g (\eul(\alpha,\beta,\cA_{\alpha,\beta} \, \zb\xi))
    \d \mu(\zb \xi) \d u_{\S^2}(\sph(\alpha,\beta)) 
    =
    \inn{\mu, \mathcal W^*g}.
  \end{equation}
  The second identity follows analogously.
\end{proof}

\begin{remark}
  The restricted semicircle transform $\W_{\alpha,\beta}: \M(\S^2) \to \M(\T)$
  is indeed related to the $L^p$ adjoint $W_{\alpha,\beta}^*$ in \eqref{eq:Bab*}.
  Although the integral on the right-hand side of \eqref{eq:Wab-meas-adj}
  is always well defined, i.e., 
  $W_{\alpha,\beta}^* \, g \in L^\infty(\S^2)$,
  the integral is no dual pairing in the measure/continuous function sense
  since $W_{\alpha,\beta}^* \, g$ is discontinuous at the zenith $\sph(\alpha,\beta)$ in general.
  Therefore, unlike in \autoref{prop:V-as-adj} for the vertical slice transform, 
  equation \eqref{eq:Wab-meas-adj} does not constitute a proper definition of $W_{\alpha,\beta}$ for measures via the dual pairing.
\end{remark}

The definitions
of the semicircle transform in the function and measure setting
are consistent 
in the sense that both coincide for absolutely continuous measures.

\begin{proposition}
  For $f \in L^1(\S^2)$,
  the semicircle transforms satisfy
  \begin{equation}
    \W [f \sigma_{\S^2}] = (\W f) \, \sigma_{\SO}
    \quad\text{and}\quad
    \W_{\alpha,\beta}[f \sigma_{\S^2}] = (\W_{\alpha,\beta} f) \, \sigma_\T.
  \end{equation}
  In particular,
  the transformed measures are again absolutely continuous.
\end{proposition}

\begin{proof}
  Both identities directly follow from \autoref{prop:W-adj-meas}
  in analogy to the proof of \autoref{cor:abs1}.
  For $\W_{\alpha,\beta}$ the second dual pairing has to be replaced by an integral.
\end{proof}

By the following theorem, the injectivity of $\W\colon L^2(\S^2) \to L^2(\SO)$ generalizes to $\W\colon\M(\S^2) \to \M(\SO)$,
which also implies the injectivity of $\W\colon L^p(\S^2) \to L^p(\SO)$.

\begin{theorem} \label{thm:B_inj}
  The semicircle transform
  $\mathcal W \colon \M(\S^2)\to \M(\SO)$ defined by \eqref{def:sctm} is injective.
\end{theorem}

The proof is given in \autoref{app:abc}.

\section{Spherical Sliced Wasserstein Distances} \label{sec:SOT}
The computation of the Wasserstein distance on the sphere
consists in determining a transport plan
between the considered probability measures.
To avoid the occurring optimization problem,
the general idea behind so-called sliced Wasserstein distances
\cite{KolNadSimBadRoh19, Bon22} is
to transform the measures first to one-dimensional domains,
and to exploit the explicit solution formula of the one-dimensional transport.
Based on the vertical slice and the normalized semicircle transform,
we can define two kinds of spherical sliced distances.
For $p \in [1, \infty)$ and $\mu, \nu \in \M(\S^2)$, we define the \emph{vertical sliced Wasserstein distance}
\begin{equation} 
    \label{eq:VSW}
    \VSW_p^p(\mu,\nu)
    \coloneqq
    \int_{\T} 
    \WS_p^p( \V_{\psi}\mu , \V_{\psi}\nu ) \d \psi,
\end{equation}
and the \emph{semicircular sliced Wasserstein distance}
\begin{equation} \label{eq:SSW}
    \SSW_p^p(\mu,\nu)
    \coloneqq
    \int_{\S^2}
    \WS_p^p( \W_{\alpha,\beta}\mu , \W_{\alpha,\beta}\nu )
    \d \sigma_{\S^2}(\sph(\alpha,\beta)),
\end{equation}
which are integrals over Wasserstein distances on  $\II$ and $\T$, respectively.

\begin{theorem}
    For every $1 \le p < \infty$,
    the vertical sliced Wasserstein distance $\VSW_p$ is a metric on $\M_{\sym}(\S^2)$, which was defined in \eqref{eq:Msym},
	and the semicircular Wasserstein distance $\SSW_p$
    is a metric on $\M(\S^2)$.
\end{theorem}

\begin{proof}
    The symmetry and the triangle inequality follow from
    the corresponding properties of the Wasserstein distance
    and the $p$-norm on $\T$ and $\S^2$.
    The positive definiteness follows from the injectivity of $\V$ and $\W$
    in \autoref{thm:V_inj} and \autoref{thm:B_inj}.
\end{proof}

Since the geodesic distance $d(\zb\xi,\zb\eta) = \arccos(\zb\xi^\top\zb\eta)$ on the sphere $\S^2$ is rotationally invariant,
i.e., $d(\zb Q \zb\xi, \zb Q \zb\eta) = d(\zb\xi,\zb\eta)$ for all $\zb Q \in\SO$,
the Wasserstein distance \eqref{eq:Wp} on $\S^2$ inherits this property,
i.e., 
$\WS_p(\mu,\nu) = \WS_p(\mu\circ \zb Q,\nu\circ \zb Q)$
 for all $\zb Q\in\SO$.
 The vertical sliced Wasserstein distance is only 
 partially rotation invariant.

\begin{proposition}
    For any $p\in[1,\infty)$, the vertical sliced Wasserstein distance $\VSW_p$ is
    invariant with respect to rotations \eqref{eq:R3R2} around the vertical axis,
    i.e.,
    for all $\mu, \nu \in \M(\S^2)$ and $\alpha \in \T$, it holds
    \begin{equation}
        \VSW_p(\mu, \nu) 
        =
        \VSW_p(\mu \circ \zb R_3(\alpha), \mu \circ \zb R_3(\alpha)).
    \end{equation}
\end{proposition}

\begin{proof}
    We have
    \begin{align}
        \VSW_p^p(\mu \circ \zb R_3(\alpha), \mu \circ \zb R_3(\alpha))
        &=
        \int_{\T} 
        \WS_p^p( \V_{\psi}[\mu \circ \zb R_3(\alpha)] , \V_{\psi}[\nu\circ \zb R_3(\alpha)] ) 
        \d \psi
        \\
        &=
        \int_{\T} 
        \WS_p^p( [\cS_\psi \circ \zb R_3(\alpha)^\tT]_\#\mu, [\cS_\psi \circ \zb R_3(\alpha)^\tT]_\# \nu ) 
        \d \psi.
        \end{align}
Since
    $[\cS_\psi \circ \zb R_3(\alpha)^\tT](\bxi) 
    = 
    \langle \bxi, \zb R_3(\alpha) \, (\cos \psi, \sin \psi,0)^\tT \rangle
    =
    \cS_{\psi + \alpha}(\bxi)$,
we further obtain
\begin{equation*}
        \VSW_p^p(\mu \circ \zb R_3(\alpha), \mu \circ \zb R_3(\alpha))      
        =
        \int_{\T} 
        \WS_p^p( \V_{\psi + \alpha} \mu , \V_{\psi + \alpha} \nu ) 
        \d \psi
        =
        \VSW_p^p(\mu, \nu).\qedhere
    \end{equation*}
\end{proof}

In contrast to the vertical sliced Wasserstein distance,
the semicircular sliced Wasserstein distance is invariant to general rotations.

\begin{proposition}
    For any $p\in[1,\infty)$, the semicircular sliced Wasserstein distance $\SSW_p$ is
    rotationally invariant,
    i.e.,
    for every $\mu, \nu \in \M(\S^2)$ and $\zb Q \in \SO$, it holds
    \begin{equation}
        \SSW_p(\mu, \nu) 
        =
        \SSW_p(\mu \circ \zb Q, \mu \circ \zb Q).
    \end{equation}
\end{proposition}

\begin{proof}
    For $\gamma \in \T$,
    let $\cT_\gamma \colon \T \to \T$
    be the shift operator given by 
    $\cT_\gamma(\psi) \coloneqq \psi - \gamma$.
    The key observation to show the statement is the identity
    \begin{equation}
        \label{eq:shift-azi}
        \cT_\gamma \circ \cA_{\alpha,\beta}(\bxi)
        =
        \azi(\eul(\alpha,\beta,0)^\tT \, \bxi) - \gamma
        =
        \azi(\eul(\alpha,\beta,\gamma)^\tT \, \bxi).
    \end{equation}
    Exploiting the shift invariance of the Wasserstein distance on $\T$,
    the identity \eqref{eq:shift-azi},
    and the rotation invariance of the surface measure on $\SO$,
    we have
    \begin{align}
        &\SSW_p^p(\mu \circ \zb Q, \nu \circ \zb Q)
        =
        \int_{\S^2}
        \WS_p^p( \W_{\alpha,\beta}[\mu \circ \zb Q] , 
        \W_{\alpha,\beta}[\nu \circ \zb Q] )
        \d \sigma_{\S^2}(\sph(\alpha,\beta))
        \\
        &\quad=
        \frac{1}{2\pi} 
        \int_\T \int_{\S^2}
        \WS_p^p( [\cT_\gamma \circ \cA_{\alpha,\beta}]_\# (\mu \circ \zb Q) , 
        [\cT_\gamma \circ \cA_{\alpha,\beta}]_\#(\nu \circ \zb Q) )
        \d \sigma_{\S^2}(\sph(\alpha,\beta)) \d \gamma
        \\
        &\quad=
        \frac{1}{2\pi} 
        \int_{\SO}
        \WS_p^p( [\azi(\eul(\alpha,\beta,\gamma)^\tT \zb Q^\tT ]_\# \mu, 
        [\azi(\eul(\alpha,\beta,\gamma)^\tT \zb Q^\tT \cdot)]_\# \nu )
        \d \sigma_{\SO}(\eul(\alpha,\beta,\gamma))
        \\
        &\quad=
        \frac{1}{2\pi} 
        \int_{\SO}
        \WS_p^p( [\azi(\eul(\tilde\alpha,\tilde\beta,\tilde\gamma)^\tT \cdot)]_\# \mu, 
        [\azi(\eul(\tilde\alpha,\tilde\beta,\tilde\gamma)^\tT \cdot)]_\# \nu )
        \d \sigma_{\SO}(\eul(\tilde\alpha,\tilde\beta,\tilde\gamma))
        \\
        &\quad=
        \frac{1}{2\pi} 
        \int_\T \int_{\S^2}
        \WS_p^p( [\cT_{\tilde\gamma} \circ \cA_{\tilde\alpha,\tilde\beta}]_\# \mu, 
        [\cT_{\tilde\gamma} \circ \cA_{\tilde\alpha,\tilde\beta}]_\#\nu)
        \d \sigma_{\S^2}(\sph(\tilde\alpha,\tilde\beta)) \d \tilde\gamma
        =
        \SSW_p^p(\mu,\nu). \qedhere
    \end{align}
\end{proof}

\section{Discrete Spherical Transforms and Inversion}\label{sec:discrete}

\subsection{Discretization and Inversion via Moore--Penrose Pseudoinverse}
\label{sec:disc-and-inv}
We will compute the sliced spherical transforms of (probability density) functions
numerically based on the singular value decomposition in a similar way as in \cite{HiPoQu18}.
To this end, we need an appropriate discretization.
In particular we need quadrature formulas on $\S^2$ as well as on the image domains $\T \times \II$ 
and $\SO$ of $\V$ and $\W$, respectively.

Let $N\in\N$.
We choose quadrature nodes $\bxi^m\in\S^2$ and respective weights $w_m>0$, $m \in [M] \coloneqq \{1,\dots,M\}$
such that all  spherical harmonics of degree $\le 2N$ are exactly integrated by the corresponding quadrature rule, see \cite{GrPo10,HiQu15}.
To be more precise,
we use the equispaced nodes $\varphi_i = i \pi/(N+1)$, $i\in[2N+2]$,
and the Gauss--Legendre nodes $t_j \in [-1,1]$, $j\in[N+1]$,
given by the roots of the $(N+1)$st Legendre polynomial.
We denote the corresponding Gauss--Legendre weights by $r_j$.
Now we obtain the quadrature
\begin{equation}
    \bxi^{m(i,j)} \coloneqq \sph(\varphi_i, \arccos t_j)
    \quad\text{and}\quad
    w_{m(i,j)} \coloneqq 2\pi r_j / (2N+2)
,\qquad 
i\in[2N+2],\;
j\in[N+1],
\end{equation}
where $m(i,j)\in [M]$ denotes the index related to the pair $(i, j)$ and $M=2(N+1)^2$.
Setting 
\begin{equation}
    \label{eq:f-samples}
    \zb f
    \coloneqq (f_m)_{m=1}^M
    \coloneqq (f(\bxi^m))_{m=1}^M, 
    \quad
    \zb Y_n^k \coloneqq (Y_n^k(\bxi^m))_{m=1}^M,
    \quad\text{and}\quad 
    \zb w  \coloneqq (w_m)_{m=1}^M,
\end{equation}
we approximate the spherical harmonics coefficients 
$\inn{f, Y_n^k}$ by
$$
\inn{\zb f, \zb Y_n^k}_{\bw}
\coloneqq
\sum_{m=1}^{M} f(\bxi^m)\, \overline{Y_n^k(\bxi^m)}\, w_m,
\qquad  n=0,\dots,N,\ k=-n,\dots,n.
$$
In particular, we have that $\inn{f, Y_n^k} = \inn{\zb f, \zb Y_n^k}_{\bw}$ 
if $f$ is a spherical polynomial of degree $\le N$.
All discrete Fourier coefficients
can be computed efficiently in $\bigo(N^2 \log^2 N)$ arithmetic operations
utilizing the \emph{nonuniform fast spherical Fourier transform} (NFSFT) \cite{kupo02,PlPoStTa18}. 

\paragraph{Discrete Vertical Slice Transform}
For discretizing $\T\times\II$, we use again
equispaced nodes $\psi_i = i \pi / (N+1)$, $i\in[2N+2]$
and 
Gauss--Legendre nodes $t_j$ and weights $r_j$, $j\in[N+1]$.
We denote the respective quadrature weights on $\T\times\II$ by $\tilde{\zb w}$, where $\tilde w_{\ell(i,j)} = \pi r_j / (N+1)$ and $\ell(i, j) \in [L]$ with $L=2(N+1)^2$ denotes the index related to the pair $(i, j)$.
The quadrature is exact of degree $2N$,
i.e., for all linear combinations of basis functions~$B_n^k$ with $0\le n,\abs{k} \le 2N$.
Using the singular value decomposition \eqref{eq:V_svd},
we discretize $\V$ by
\begin{equation} \label{eq:Vf-d}
\V_D\colon \R^{M} \to \R^{L},\quad
\V_D\zb f 
\coloneqq
\sum_{n=0}^{N} 
\sum_{\substack{k=-n\\n+k \,  \text{even}} }^n 
\mathrm{v}_n^k\, \inn{\zb f, \zb Y_n^k}_{\bw} \,\zb B_n^k,
\end{equation}
where 
$\zb B_n^k \coloneqq (B_n^k(\psi_i,t_j))_{\ell(i,j)=1}^L$ 
and $B_n^k$ is given in \eqref{eq:Bnk}.
Then  $\V_D\zb f$ can be computed
using the fast Fourier transform (FFT) in $\psi$ 
and a fast polynomial transform \cite{postta98} in $t$ in $\bigo(N^2 \log^3 N)$ arithmetic operations.
Based on the quadrature for $\T \times \II$,
we analogously discretize the (truncated) Moore--Penrose pseudoinverse 
\eqref{MP_V} by
\begin{equation}\label{MP_V_approx}
  \V^\dagger_D \colon \R^L \to \R^M, \quad
  \V^\dagger_D \zb g 
  \coloneqq
	\sum_{n =0} ^N
  \sum_{\substack{k=-n\\n+k \,  \text{even}} }^n
  \frac{1}{{\mathrm v}_n^k} \,
  \Bigl\langle\zb g, \,\zb B_n^k \Bigr\rangle_{\tilde {\zb w}} \,
  \zb Y^k_n,
\end{equation}
where $\zb g \coloneqq (g(\psi_i,t_j))_{\ell(i,j)=1}^{L}$ 
consists of samples of $g\colon \T\times\II \to\R$.
For a spherical polynomial $f$ of degree $N$, 
the chosen quadratures ensure $\V^\dagger_D\V_D \zb f = \zb f$.

\paragraph{Discrete Semicircle Transform}
We use quadrature nodes $\bQ_\ell\in\SO$ and weights $\tilde{\zb w} = (\tilde w_{\ell})_{\ell=1}^L$
such that all 
rotational harmonics of degree $\le 2N$ are exactly integrated.
Since it becomes clear from the context which weights are addressed, we use again $\tilde{\zb w}$.
In particular, we consider a quadrature \cite{GrPo09} on $\SO \cong \S^2\times\S^1$ as product of a Gauss-type quadrature on $\S^2$, see \cite{GrPo10}, and an equispaced quadrature on $\T$.
We use this product structure because we can now discretize $\W_{\alpha,\beta}$ on a uniform grid.
Similarly to \eqref{eq:Vf-d}, the singular value decomposition \eqref{eq:9} of $\W\colon L^2(\S^2) \to L^2(\SO)$ can be truncated as
\begin{align} \label{eq:Bf-d}
\mathcal W_D\colon \R^M \to \R^L, \quad
\mathcal W_D\zb f 
&\coloneqq
\sum_{n=0}^N \sum_{j,k=-n}^{n}
\lambda_n^j\, \inn{\zb f, \zb Y_n^k}_{\bw}\, \overline{\zb D_n^{k,j}}, 
\end{align}
where $\zb D_n^{k,j} \coloneqq (D_n^{k,j}(\bQ_\ell))_{\ell=1}^L$.
Then \eqref{eq:Bf-d} can be computed in $\bigo(N^3 \log^2 N + L)$ arithmetic operations
with the \emph{nonuniform fast SO(3) Fourier transform} (NFSOFT) \cite{poprvo07}.
Further, 
we approximate the Moore--Penrose pseudoinverse \eqref{eq:Bdag} of $\mathcal W$ 
 by
\begin{equation} \label{eq:Bdag-d}
    \mathcal W_D^\dagger \colon \R^L\to\R^M, \quad
    \mathcal W_D^\dagger \zb g
    \coloneqq
    \sum_{n=0}^N \sum_{j,k=-n}^{n}
    \frac{1}{(\mathrm{w}_n)^2}\, \lambda_n^j \inn{\zb g, \overline{\zb D_n^{k,j}}}_{\tilde{\zb w}}\,\zb Y_n^k,
\end{equation}
where $\zb g \coloneqq (g(\bQ_\ell))_{\ell=1}^L $ for $g\colon \SO\to\R$.
As above, $\W_D^\dag \zb g$ can be evaluated with NFSFT and NFSOFT algorithms.
For a spherical polynomial $f$ of degree $N$,
we have $\W^\dagger_D \W_D \zb f = \zb f$.

\subsection{Inversion by Variational Approach} \label{sec:reg}
The push-forward definitions of
the sliced spherical transforms 
ensure that
probability measures are mapped to probability measures.
In the context of optimal transport,
we require that the inverse transforms have the same behaviour.
Even when restricting to the function setting,
we can however construct functions
with non-trivial negative part
that are transformed into probability densities,
e.g.\ by taking a function that is negative in a sufficiently small spherical cap $\{\bxi\in\S^2 : \xi_3<1-\varepsilon\}$ 
and equals a positive constant otherwise.
Thus the Moore--Penrose pseudoinverse applied to a probability density 
is not necessarily a probability density.
To overcome this issue, 
we consider the inversion of the discretized spherical transforms as inverse problems, which we solve using a variational formulation.

As in \eqref{eq:f-samples},
let $\zb f \in \R^M$ contain the samples of 
the probability density function $f$ on $\S^2$,
and $\zb w \in \R^M$
the respective quadrature weights.
If the quadrature is exact for $f$,
we have
$\zb f\ge 0$ and 
$\int_{\S^2} f \d \sigma_{\S^2} 
= \sum_{m=1}^M w_m f_m = 1$.
Thus $\zb f$ can be interpreted as
probability density function
with respect to the counting measure weighted by $\bw$.
For the numerical inversion,
we handle both transforms simultaneously, 
denoting the discretizations $\V_D$ and $\W_D$ 
by $\mathcal T_D \colon \R^M \to \R^L$.
Let $\zb g \in \R^{L}$ be the samples of the density function $g$ on $\T \times \II$ or $\SO$.
We equip $\R^M$ with the weighted Euclidean inner product $\langle \zb f, \tilde{\zb f} \rangle_{\zb w} 
\coloneqq \sum_{m=1}^M w_m f_m \tilde f_m$,
and, analogously, $\R^L$ with 
the inner product 
$\langle \zb g, \tilde{\zb g} \rangle_{\tilde{\zb w}} 
\coloneqq \sum_{\ell=1}^L \tilde w_\ell g_\ell \tilde g_\ell$,
where $\tilde{\zb w}$ contains the quadrature weights for $\T \times \II$ or $\SO$.
Furthermore, we denote the all-one vector by $\zb 1$.

Now, we aim to find
an approximate solution
$\zb f$ of the inversion problem
$\mathcal T_D \zb f = \zb g$ in the weighted probability simplex
$\Delta_\bw \coloneqq \{ \zb f \in \R^M : \zb f \ge 0, \allowbreak \langle \zb f, \zb 1 \rangle_\bw = 1\}$.
To this end,
we introduce a regularized inverse as
the minimizer of the strictly convex optimization problem
\begin{equation}
    \label{eq:reg-inv}
    \argmin_{\zb f \in \Delta_\bw} \KL_{\tilde{\bw}}(\mathcal T_D \zb f, \zb g) + \rho \KL_{\bw}(\zb f, \zb 1), \quad \rho > 0,
\end{equation}
where 
$\KL_\bw$ is the discrete \emph{Kullback--Leibler (KL) divergence} on the weighted space $\mathbb R^M$ given by
\begin{equation}   \label{eq:KL-w}
  \KL_\bw( \zb f, \tilde{\zb f})
  \coloneqq
  \langle \zb f, \log \zb f - \log \tilde{\zb f} \rangle_\bw + \langle \tilde{\zb f} -\zb f, \zb 1 \rangle_{\zb w},
 \end{equation}
for $\zb f, \tilde{\zb f} \ge 0$ with $f_m = 0$ whenever $\tilde f_m = 0$,
and $\KL_\bw( \zb f, \tilde{\zb f}) \coloneqq +\infty$ otherwise.
Here the logarithm acts componentwise, 
and we set $0 \log 0 \coloneqq 0$.
Note that $\KL_{\bw}(\zb f, \zb 1)$ is the negative entropy of $\zb f$.
The KL divergence on $\R^L$ is defined analoguously.

To find the minimizer of \eqref{eq:reg-inv},
we employ the primal-dual splitting of Chambolle and Pock \cite{ChPo16a}.
To this end, we reformulate \eqref{eq:reg-inv} as
\begin{equation}     \label{eq:reg-inv_1}
    \argmin_{\zb f, \zb y_1,\zb y_2} \KL_{\tilde{\bw}}(\zb y_1, \zb g) + \rho \KL_{\bw}(\zb y_2, \zb 1) + \chi_{\Delta_\bw} (\zb f)
		\quad\text{s.t.}\quad \mathcal T_D \zb f = \zb y_1,\, \zb f = \zb y_2
\end{equation}
with the \emph{characteristic function}  $\chi_{\Delta_\bw}(\zb f) = 0$ for $\zb f \in \Delta_\bw$ and $\chi_{\Delta_\bw}(\zb f) = +\infty$ else.
For $\theta \in (0,1]$ and $\sigma, \tau >0$ such that $1/\tau \sigma > \lVert I + \mathcal T_D^* \mathcal T_D \rVert$,
the algorithm converges and reads as  
\mathtoolsset{showonlyrefs=false}%
\begin{subequations} \label{eq:PD}
  \begin{align}
    \label{eq:PD:primal}
    \zb f^{k+1}
    &\coloneqq
      \proj_{\Delta_\bw} \bigl(\zb f^k - \tau \mathcal T_D^*\zb y^k_1 - \tau \zb y^k_2 \bigr),\\
    \tilde{\zb f}^{k+1}
    &\coloneqq
      \zb f^{k+1} + \theta (\zb f^{k+1} - \zb f^k),
    \\
    \zb y_1^{k+1} 
		&\coloneqq 
		\prox_{\sigma \KL^*_{\tilde{\zb w}}(\cdot,\zb g) } \bigl(\zb y_1^k + \sigma \mathcal T_D \tilde{\zb f}^{k+1}\bigr)
		\\
    \zb y_2^{k+1}
    &\coloneqq
		\prox_{\sigma (\rho \KL_{\zb w})^*(\cdot,\zb 1) } \bigl(\zb y_2^k + \sigma \, \tilde{\zb f}^{k+1}\bigr).
  \end{align}
  \end{subequations}
  \mathtoolsset{showonlyrefs=true}%
  Here $\proj_{\Delta_\bw}$ is the orthogonal projection onto $\Delta_\bw$.
  Further,
    for a function $h \colon \R^M \to \R$,
    the \emph{proximal operator} with respect to the weight $\bw$ is given by
    \begin{equation}
        \label{eq:prox}
      \prox_{\sigma h}({\zb x})
      \coloneqq
      \argmin_{{\zb y} \in \R^M}
      h({\zb y}) + \tfrac{1}{2\sigma} \, \lVert \zb x - {\zb y} \rVert^2_{\zb w}
    \end{equation}
    and its \emph{Fenchel conjugate} by
    $h^* (\zb y) 
	    \coloneqq 
	    \max_{\zb x \in \R^M}
        \langle \zb x, \zb y\rangle_{{\zb w}} + h(\zb y)$.
    On $\R^L$ with weight $\tilde \bw$, 
    the proximal operator and conjugate are defined similarly.

\begin{proposition}
  \label{prop:proj-ps}
  The orthogonal projection onto $\Delta_\bw$ with respect to the inner product $\inn{\cdot,\cdot}_\bw$ is given by
  \begin{equation}
    \proj_{\Delta_\bw}(\zb f)
    = 
    [\zb f + \lambda \zb 1]_+,
  \end{equation}
  where $[\cdot]_+$ denotes the componentwise positive part, 
  and $\lambda$ is the root of
  $\langle
    \zb 1, [\zb f + \lambda \zb 1]_+
    \rangle_{\zb w} - 1$.
\end{proposition}

The statement follows line by line
via incorporating the weighted inner product into the argumentation in \cite[Thm~6.27]{Be17}.
The function in \autoref{prop:proj-ps} 
is monotonically increasing 
and piecewise linear
with finitely many pieces; 
thus
the root can be determined using a bisection method
to identify the piece with sign change
and solving a linear equation.
For the standard probability simplex with $\bw=\zb1$,
there exist several further numerically efficient approaches
\cite{Con16}.

\begin{proposition}
Let $\sigma, a > 0$, and $\zb b \in \R^M$.
On the weighted Euclidean space $\R^M$,
the KL divergence satisfies
$$
\prox_{\sigma (a \KL)^*(\cdot,\zb b)}({\zb x}) 
= \zb x -  a W \bigl( \tfrac{\sigma}{a} \, \zb b \odot
\exp \bigl(\tfrac{1}{a} \, \zb x \bigr) \bigr),
\qquad \bx\in\R^M,
$$
where $W$ denotes the componentwise applied \emph{Lambert's $W$-function} that maps $z$ to the solution $y$ of $y  \exp y = z$
and $\odot$ the componentwise multiplication.
\end{proposition}

\begin{proof}
Differentiating the objective of the weighted Fenchel conjugate
and setting to zero yields
\begin{equation}
    (a \KL_{{\zb w}})^* (\zb y, \zb b) 
    = 
    a \, \big\langle \zb b, \exp\big(\tfrac{1}{a} \, \zb y\big) \big\rangle_{{\zb w}} + a \, \langle \zb b, \zb 1 \rangle_{{\zb w}}.
\end{equation}
Inserting the conjugated scaled KL divergence into \eqref{eq:prox},
and setting the derivative again zero, 
we componentwise obtain
\begin{align*}
    \zb b \odot \exp\bigl(\tfrac{1}{a} \, \zb y \bigr)
    = 
    \tfrac{1}{\sigma} \, (\zb x - \zb y)
    &\quad\Leftrightarrow\quad
    \log(\zb b) + \tfrac{1}{a} \, \bx
    =
    \tfrac{1}{a} \, (\bx - \by) + \log \bigl(\tfrac{1}{\sigma} \, (\bx-\by) \bigr)
    \\
    &\quad\Leftrightarrow\quad
    \tfrac{\sigma}{a} \, \zb b \odot \exp\bigl(\tfrac{1}{a} \, \zb x\bigr)
    =
    \tfrac{1}{a} \, (\zb x- \zb y) \odot \exp\bigl( \tfrac{1}{a} \, (\zb x- \zb y) \bigr), 
\end{align*}
which gives the assertion.
\end{proof}

Note that the primal-dual algorithm requires the adjoint $\mathcal T_D^*$ in \eqref{eq:PD:primal}.
Based on the discretized spherical transforms in \autoref{sec:disc-and-inv},
we obtain their adjoint operators 
\begin{equation}
\V_D^* \bg
=
\sum_{n =0} ^N
\sum_{\substack{k=-n\\n+k \,  \text{even}} }^n
{\mathrm v}_n^k \,
\inn{ \zb g, \zb B_n^k }_{\tilde {\bw}} \,
\zb Y^k_n \quad\text{and}\quad
    \W_D^* \zb g
    =
    \sum_{n=0}^N \sum_{k,j=-n}^{n}
    \lambda_n^j 
    \inn{\zb g, \overline{\zb D_n^{k,j}}}_{\tilde\bw} \zb Y_n^k.
\end{equation}
The primal-dual iteration \eqref{eq:PD} may be summarized as follows. 

\begin{algorithm}[Primal-Dual for Regularized Inversion]
  \label{alg:class-PD}
  {\,}\newline
  \emph{Input:} $\bg \in \R^L$, $\theta \in (0,1]$, $\sigma, \tau, \rho >0$.
  \newline
  \emph{Initialization:}
  $\zb f^{0} \coloneqq (4\pi)^{-1}\, \zb1$, 
  $\by^0_1 \coloneqq \zb0$, $\by^0_2 \coloneqq \zb0$.
  \newline
  \emph{Iteration:}
  For $k=0,1,\dots$ until convergence do
  \begin{enumerate}[(a), nosep]
    \item $\zb f^{k+1}
    \coloneqq
      \proj_{\Delta_\bw} (\zb f^k - \tau \mathcal T_D^* \zb y^k_1 - \tau \zb y^k_2)$,
    \item
      $\tilde{\zb f}^{k+1}
    \coloneqq
      \zb f^{k+1} + \theta (\zb f^{k+1} - \zb f^k)$,
    \item
    $\tilde{\zb y}_1^{k+1}
    \coloneqq 
      \zb y_1^k + \sigma \mathcal T_D \tilde{\zb f}^{k+1}$,
    \item
      $\zb y_1^{k+1} 
    \coloneqq
      \tilde{\zb y}_1^{k+1} 
      - W\bigl(\sigma
      \zb g \odot
      \exp\bigl(\tilde{\zb y}_1^{k+1} \bigr)\bigr)$,
    \item
    $\tilde{\zb y}_2^{k+1}
    \coloneqq
      \zb y_1^k + \sigma \, \tilde{\zb f}^{k+1}$,
    \item
      $\zb y_2^{k+1}
    \coloneqq
      \tilde{\zb y}_2^{k+1} - \rho \, 
      W \bigl( \tfrac{\sigma}{\rho } 
      \exp\bigl( \tfrac{1}{\rho }  \tilde{\zb y}_2^{k+1} \bigr) \bigr)$.
  \end{enumerate}
  \emph{Output:} $\zb f \in \R^M$ solving \eqref{eq:reg-inv}.
\end{algorithm}

\section{Numerical Results} \label{sec:numerics}
In this section, we provide proof-of-concept examples 
that the sliced spherical transforms can be combined in a meaningful way
with optimal transport on the interval and the circle.
First, we deal with the approximation of Wasserstein barycenters on the sphere.
In particular, this requires the inversion of the sliced spherical transforms.
Second, we show that these transforms combined with optimal transport can be used
for classifying classes of measures.
All numerical tests are performed in Matlab R2022a on an Intel Core i7-10700 CPU with 16\,GB memory.

\subsection{Interpolation between Probability Measures}
\label{sec:int-meas}

Given two probability measures on the sphere,
we generate a measure \enquote{between} them,
as proposed in \cite{KolParRoh16} for the Radon transform on $\R^2$.
In particular, we compute the CDT or cCDT of their spherical transform $\V $ or $\W $,
then we interpolate in the CDT space and go back to $\S^2$ via the inverse of the CDT or cCDT and the spherical transforms.

For computing the forward, inverse, and adjoint spherical transforms, we truncate the singular value decomposition at degree $N=44$ and use the software package \cite{nfft3} for the NFSFT and NFSOFT.
We have $M = (2N+2)(N+1) = 4050$ quadrature nodes on the sphere, cf.\ \autoref{sec:disc-and-inv}.

\paragraph{Interpolation between Mises--Fisher Distributions}
As test function on $C(\S^2)$, we choose the density of the \emph{von Mises--Fisher (vMF) distribution}
\begin{equation} \label{eq:vmf}
f_{\kappa,\zb\eta}(\bxi) = c_\kappa\, \e^{\kappa \inn{\zb\eta,\bxi}}
,\quad \bxi\in\S^2,
\end{equation} 
with the mean direction $\zb\eta\in\S^2$ and the concentration $\kappa>0$,
where $c_\kappa$ is chosen such that $\int_{\S^2} f_{\kappa,\zb\eta} \d\sigma_{\S^2} = 1$.
Since $\V$ acts only on even functions,
we make our first tests with symmetrized vMF distributions via $(f_{\kappa,\zb\eta}(\xi_1,\xi_2,\xi_3) + f_{\kappa,\zb\eta}(\xi_1,\xi_2,-\xi_3))/2$, see \autoref{fig:vmf}.

\begin{figure}[htb]
    \centering
    \begin{minipage}{.32\textwidth}
    \includegraphics[width=\textwidth]{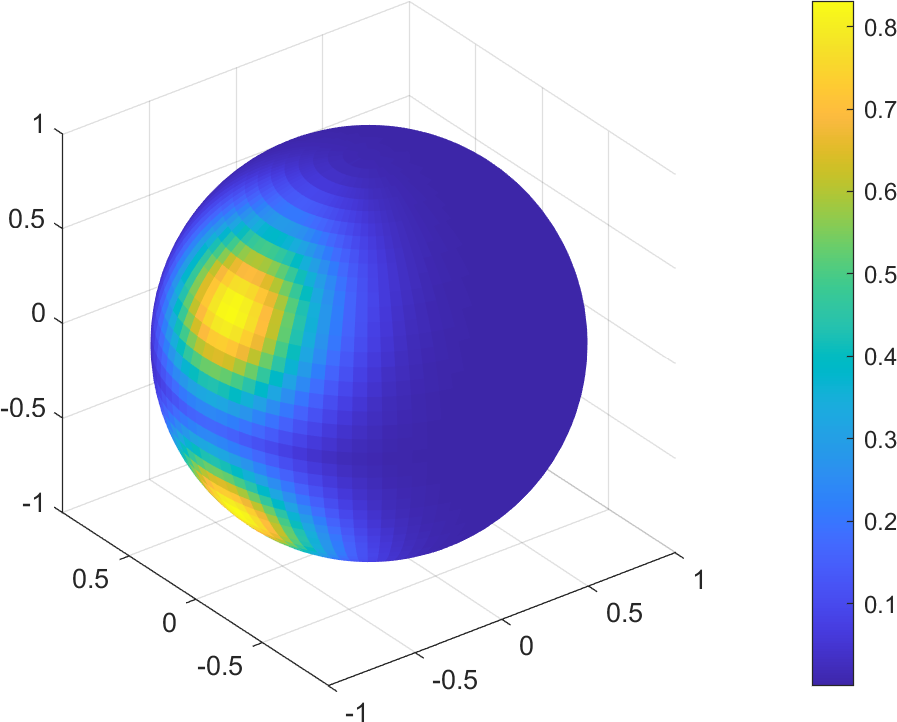}
    \end{minipage} \quad
    \begin{minipage}{.32\textwidth}
    \includegraphics[width=\textwidth]{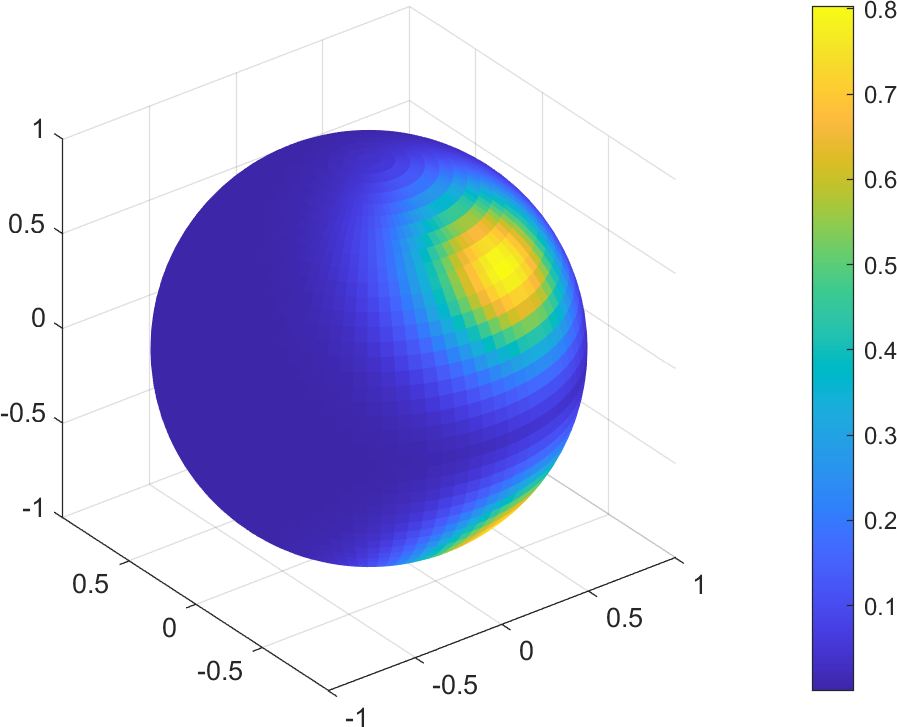}
    \end{minipage}
    
    \caption{Density functions of two symmetrized vMF distributions \eqref{eq:vmf}.}
    \label{fig:vmf}
\end{figure}

Let $\mu,\nu \in \Pac(\S^2)$ be two given measures.
For some $\delta \in[0,1]$, we set the \emph{unregularized $\V$-CDT interpolation} between $\mu$ and $\nu$ as
\begin{equation} \label{eq:V_bary}
\V^\dagger h, \quad \text{where} \quad 
h(\psi,t) = \CDT^{-1}_{\V_\psi \mu}\big[\delta\, \CDT_{\V_\psi \mu}[ \V_\psi \nu]\big](t),
\quad \psi\in\T,\ t\in\II.
\end{equation}
Here, we discretize $\mathcal V_\psi $ via \eqref{eq:Vf-d} with $2(N+1)^2=4050$ quadrature nodes on $\T\times\II$.
Our implementation of the CDT and its inverse is based on \cite{KolParRoh16}.\footnote{See the Python code \url{https://github.com/skolouri/Radon-Cumulative-Distribution-Transform}}

Analogously,
we define the \emph{unregularized $\W$-CDT interpolation} $\W^\dagger h$, where
$$
h(\eul(\alpha,\beta,\gamma))
= \CDT^{-1}_{\W_{\alpha,\beta} \mu}\big[ \delta\, \sCDT_{\W_{\alpha,\beta} \mu}[\W_{\alpha,\beta}\nu]\big](\gamma)
,\quad \eul(\alpha,\beta,\gamma)\in\SO.
$$
Here, the optimal parameter $\theta$ of \eqref{eq:Wp-circle}, which is required for the $\sCDT$,
is determined by the algorithm \cite{DelSalSob10}.\footnote{See the Matlab code \url{https://users.mccme.ru/ansobol/otarie/software.html}}
Moreover, we compute $\W\mu$ and $\W\nu$ by \eqref{eq:Bf-d},
where we use $L=118\,944$ quadrature points $\eul(\alpha,\beta,\gamma)$ on $\SO$, which are obtained as the product of a Gauss-type quadrature\footnote{Quadrature rule on $\S^2$ from \url{http://www.tu-chemnitz.de/~potts/workgroup/graef/quadrature}.} in $\sph(\alpha,\beta) \in \S^2$ 
and a uniform grid in~$\gamma$.

Instead of the Moore--Penrose pseudoinverse $\V^\dagger$ or $\W^\dagger$,
we also apply the primal-dual \autoref{alg:class-PD} to obtain the regularized inverse \eqref{eq:reg-inv} of $h$, which we call the regularized $\V$-CDT  or $\W$-CDT interpolation. Here we choose the regularization parameter $\rho=0.1$ and step sizes $\sigma=1$ and $\tau=1/4$,
and we terminate the algorithm after 200 iterations.
The CDT interpolations for $\delta=0.5$ are plotted in \autoref{fig:vmf_bary}.
While the regularization has a comparably small effect on the $\V$-CDT interpolation,
we note that the unregularized $\W$-CDT interpolation is severely negative in some areas and therefore not a probability density,
which is circumvented by the primal-dual algorithm.

\begin{figure}[htb]
    \centering
    \begin{subfigure}[t]{.32\textwidth}
    \includegraphics[width=\textwidth]{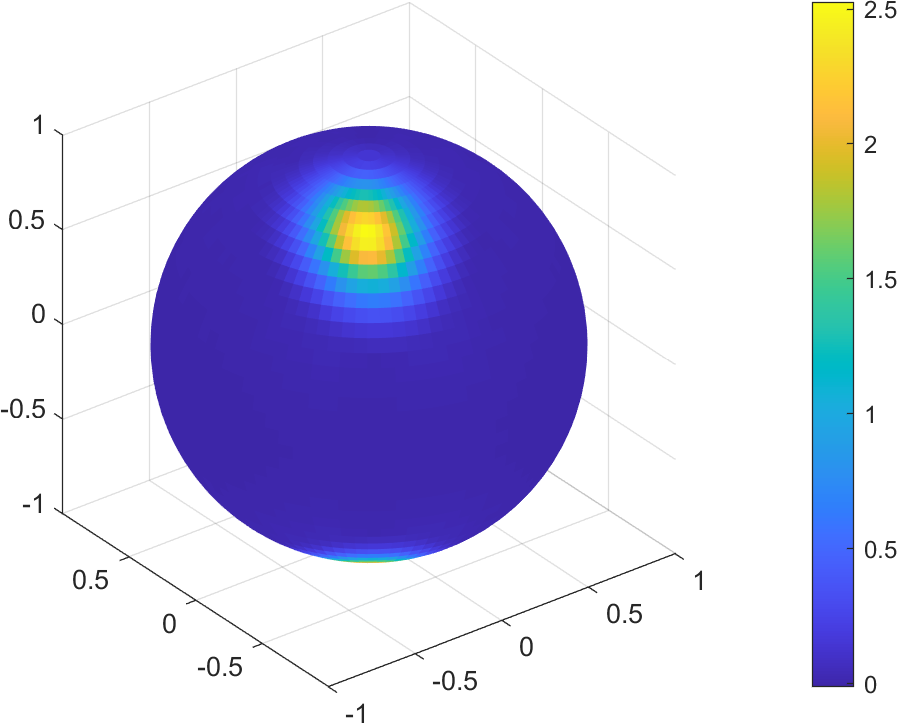}
    \caption{$\V$-CDT interpolation (0.01\,s)}
    \end{subfigure}\hfill
    \begin{subfigure}[t]{.32\textwidth}
    \includegraphics[width=\textwidth]{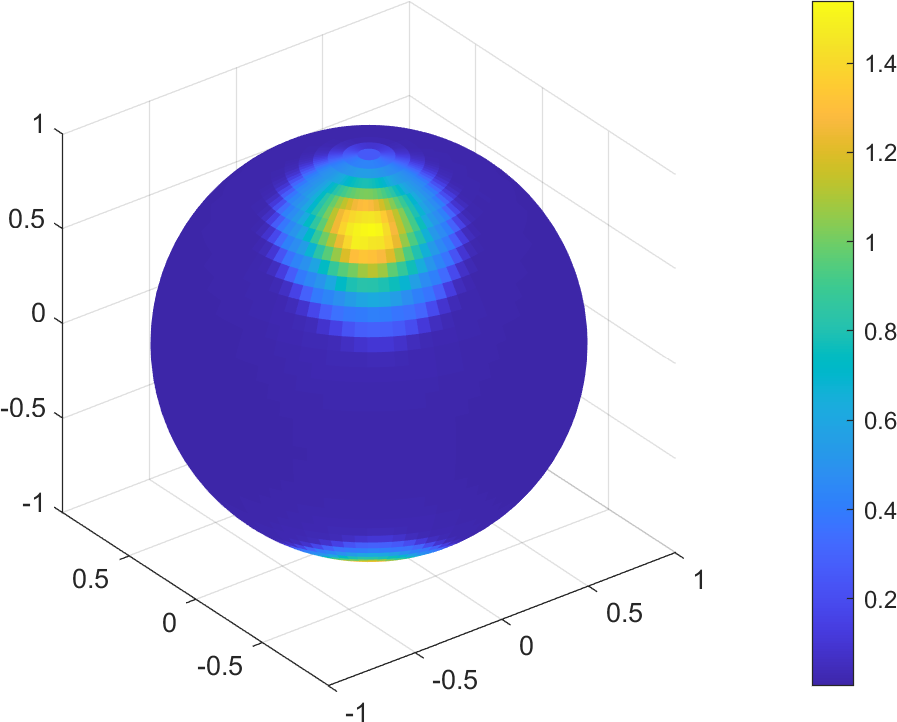}
    \caption{$\V$-CDT interpolation with regularized inverse (0.3\,s) \label{fig:vmf_vreg}}
    \end{subfigure}\hfill
    \begin{subfigure}[t]{.32\textwidth}
    \includegraphics[width=\textwidth]{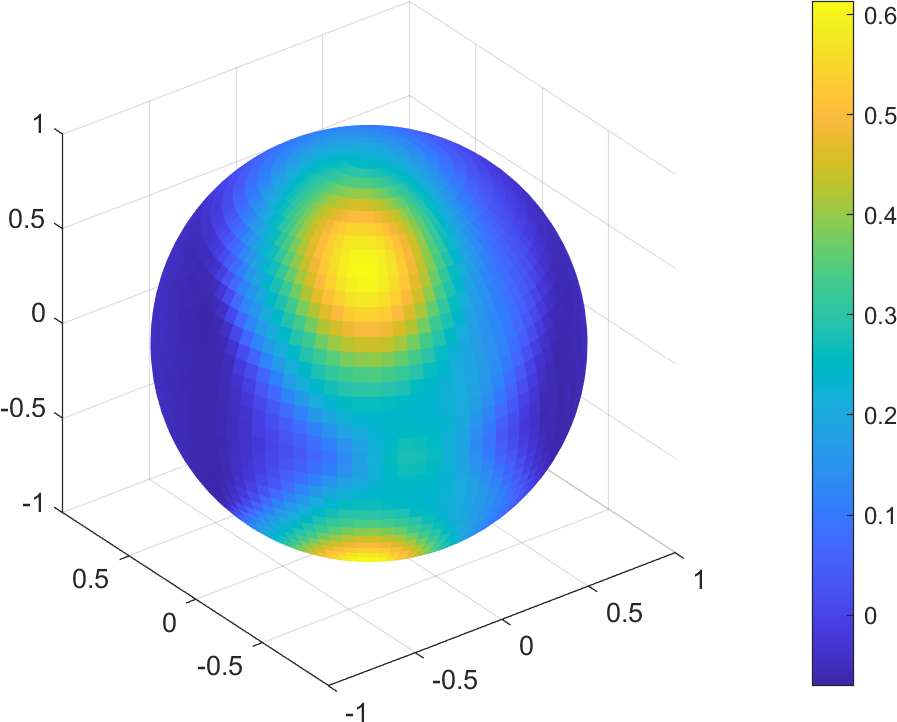}
    \caption{$\W$-CDT interpolation with Moore--Penrose pseudoinverse $\W^\dagger$ (3.2\,s)}
    \end{subfigure}
    \begin{subfigure}[t]{.32\textwidth}
    \includegraphics[width=\textwidth]{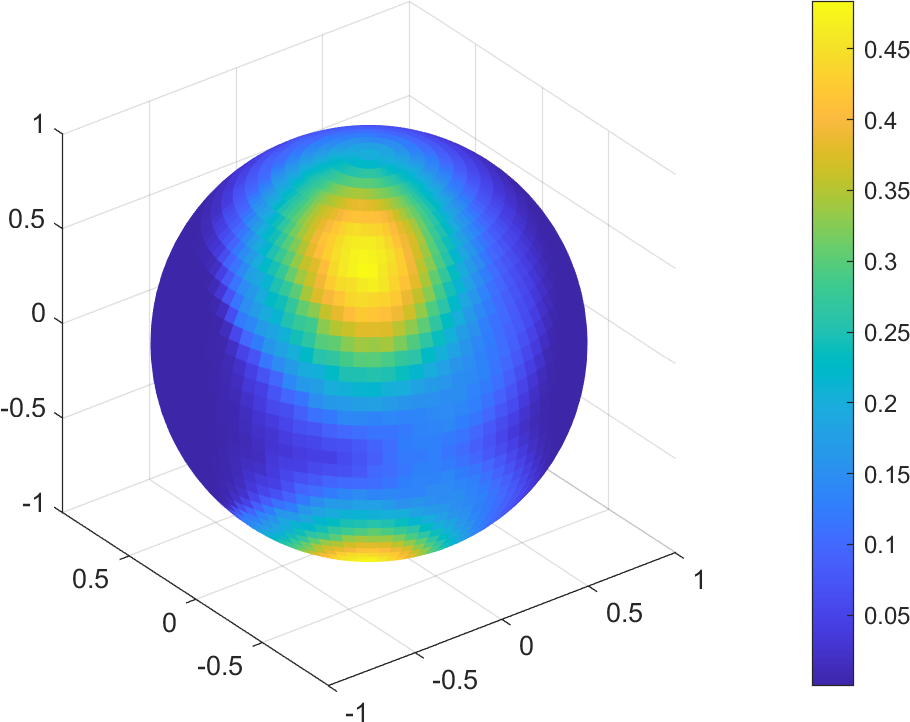}
    \caption{$\W$-CDT interpolation with regularized inverse (128\,s)}
    \end{subfigure} \hfill
    \begin{subfigure}[t]{.32\textwidth}
    \includegraphics[width=\textwidth]{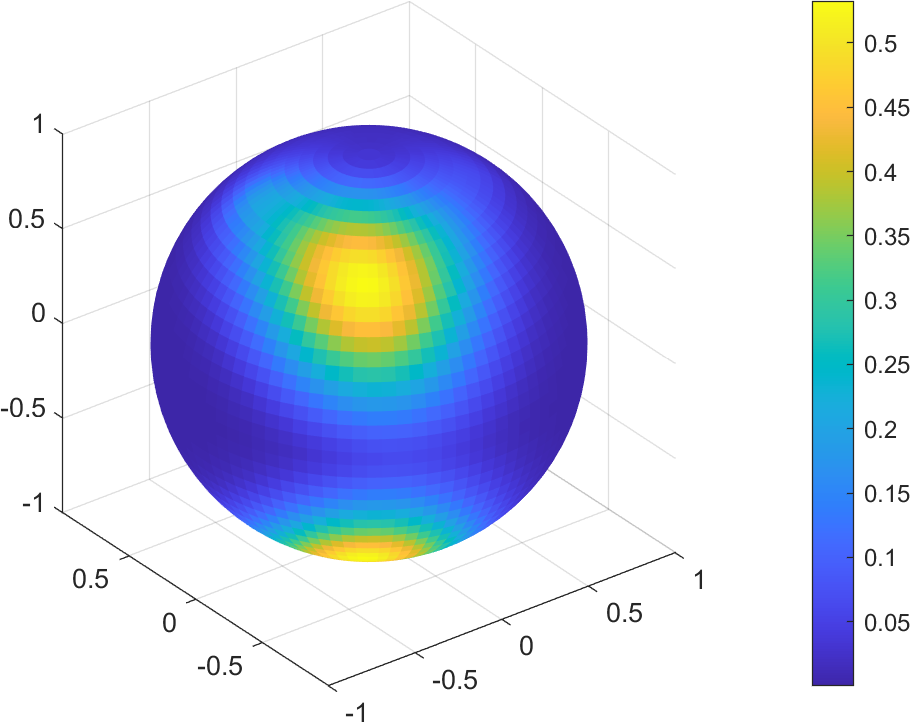}
    \caption{Regularized 2-Wasserstein bary\-center (19\,s)}
    \end{subfigure} \hfill
    \begin{subfigure}[t]{.32\textwidth}
    \includegraphics[width=\textwidth]{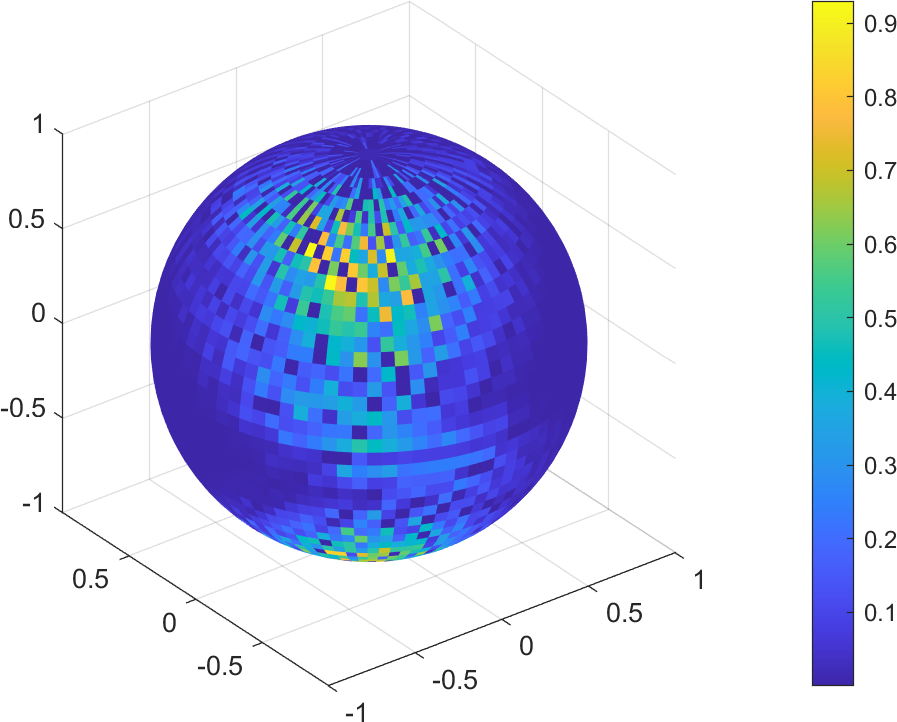}
    \caption{Unregularized 2-Wasserstein barycenter (33\,h)}
    \end{subfigure}
    
    \caption{CDT interpolation with $\delta=0.5$ of the vMF distributions from \autoref{fig:vmf}.}
    \label{fig:vmf_bary}
\end{figure}

As a reference, we consider the spherical 2-Wasserstein barycenter \eqref{eq:W-bary}
and its entropy-regularized counterpart \cite{PeyCut19}, whose computation with the Sinkhorn algorithm \cite{Kni08}
can be implemented efficiently, cf.\ \cite{BaQue22,BenCarCutNenPey15}. 
We apply the Python optimal transport library \cite{POT} for both, 
where the Sinkhorn algorithm uses the regularization parameter $0.01$ and a maximum number of 1000 iterations. 
In our example in \autoref{fig:vmf_bary}, the regularized 2-Wasserstein barycenter looks similar to the $\W$-CDT interpolation, 
while the unregularized barycenter is very noisy and takes very long to compute with a linear program solver.

\paragraph{Interpolation between vMF Distribution and a Mixture}
The $\W$-CDT interpolation of more evolved test functions, which are not symmetric, is depicted in \autoref{fig:smiley}.
We notice that the $\W$-CDT interpolation shows the \enquote{eyes} more clearly than the regularized 2-Wasserstein barycenter.
This might be caused by a too large regularization parameter of the Sinkhorn algorithm,
but when making it smaller the algorithm fails with a division by zero error.

\begin{figure}[htb]
    \centering
    \begin{subfigure}[t]{.4\textwidth} \centering
    \includegraphics[width=.8\textwidth]{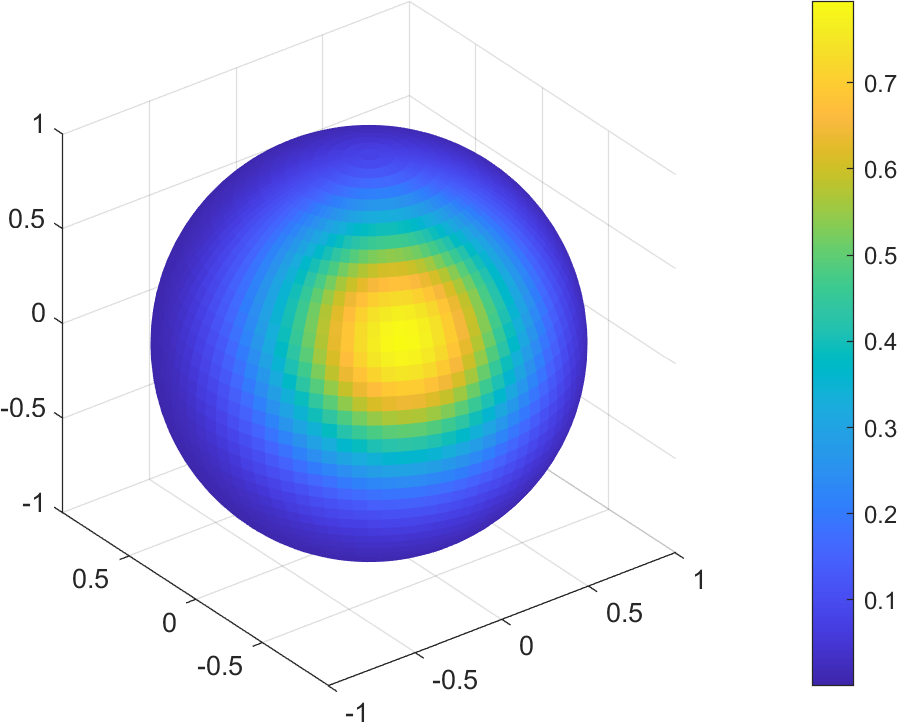}
    \caption{Density of vMF distribution $\mu$}
    \end{subfigure}
    \begin{subfigure}[t]{.4\textwidth} \centering
    \includegraphics[width=.8\textwidth]{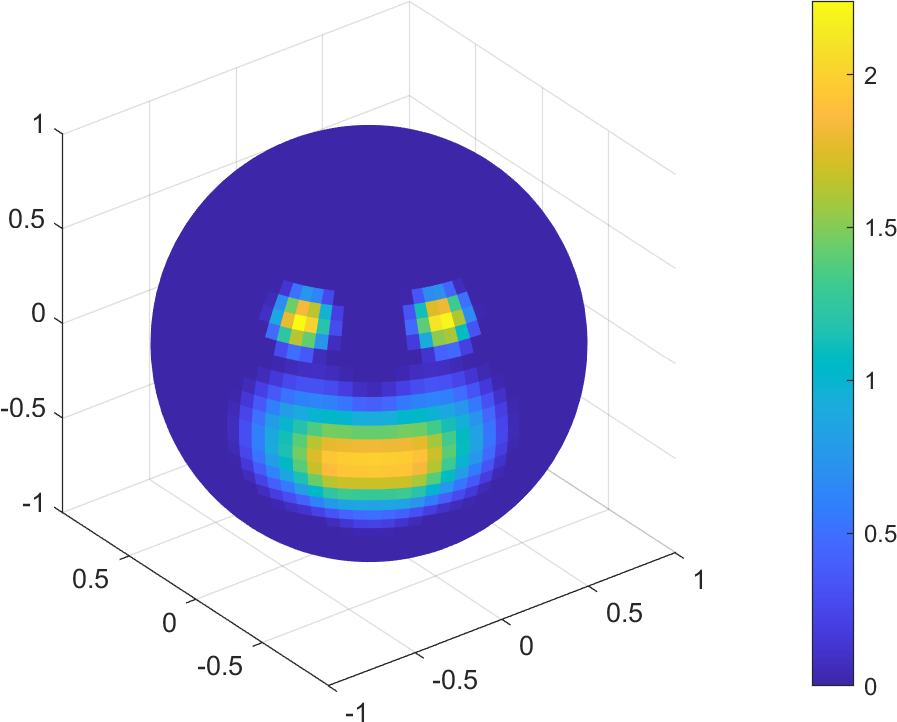}
    \caption{Density of $\nu$ (quadratic spline)}
    \end{subfigure}
    \begin{subfigure}[t]{.32\textwidth}
    \includegraphics[width=\textwidth]{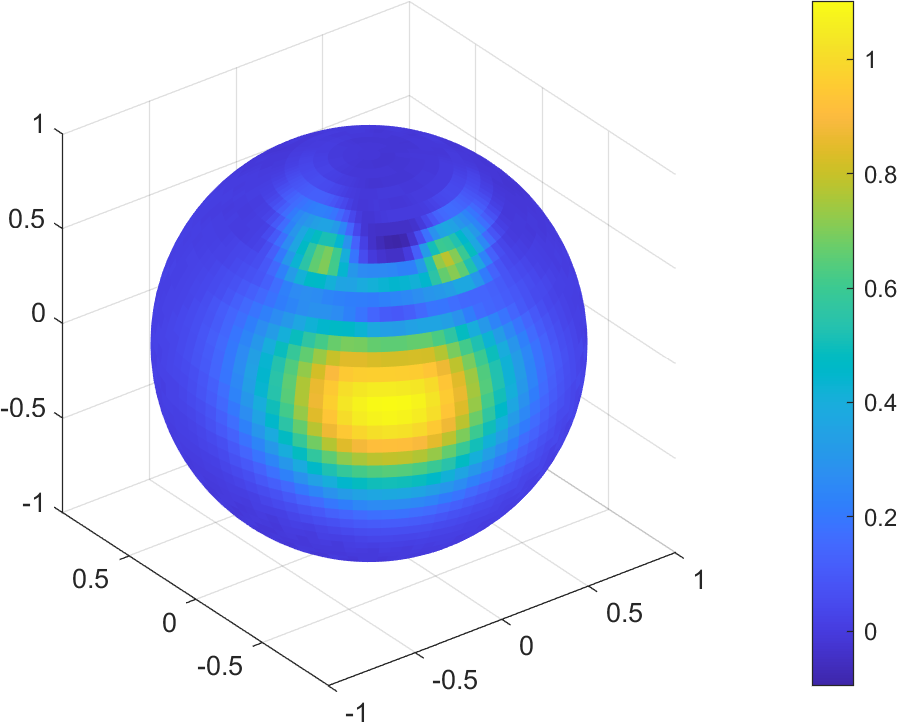}
    \caption{$\W$-CDT interpolation with Moore--Penrose pseudoinverse $\W^\dagger$}
    \end{subfigure}\hfill
    \begin{subfigure}[t]{.32\textwidth}
    \includegraphics[width=\textwidth]{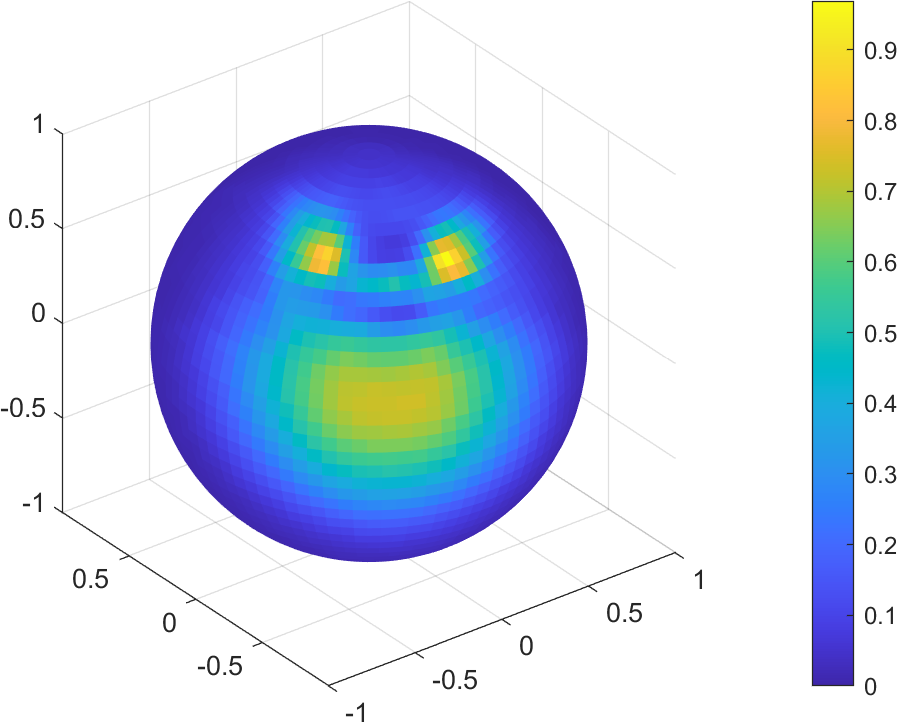}
    \caption{$\W$-CDT interpolation with regularized inverse}
    \end{subfigure} \quad
    \begin{subfigure}[t]{.32\textwidth}
    \includegraphics[width=\textwidth]{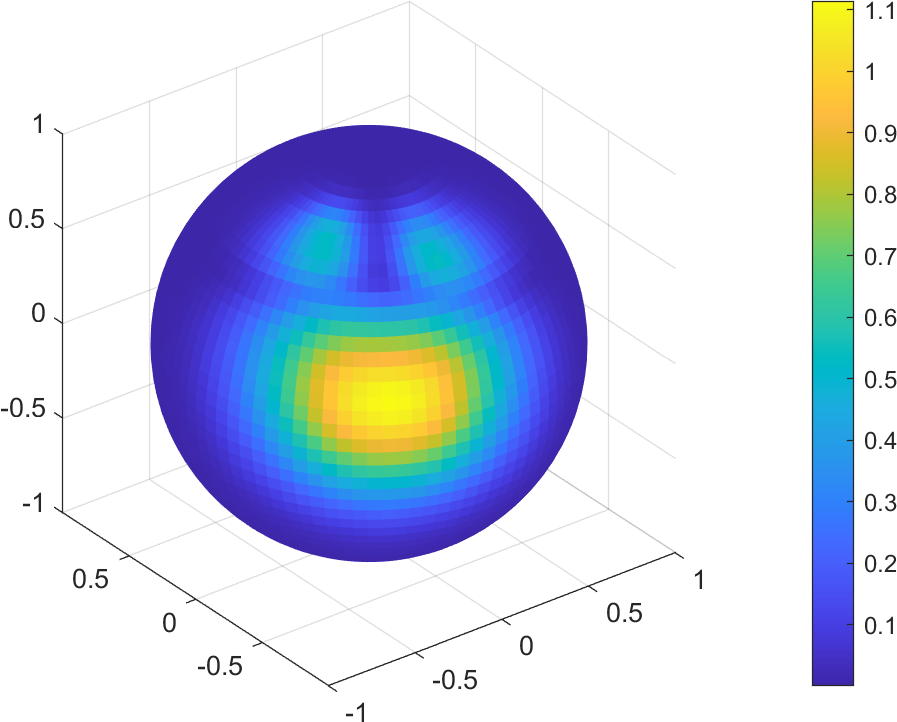}
    \caption{Regularized 2-Wasserstein barycenter }
    \end{subfigure} 
    
    \caption{CDT interpolation of density functions with $\delta=0.5$.}
    \label{fig:smiley}
\end{figure}

\subsection{Classification of Probability Measures}
\label{sec:classification}

In one dimension,
the cumulative distribution transform \eqref{eq:cdt}
is known to increase the separability 
between certain classes of probability measures.
If the considered classes are build from prototypes
using certain transformations 
like shifts or scalings,
the constructed classes are linearly separable 
in the CDT space \cite{MoCl23, PaKoSo18}.
For probability measures on multi-dimensional domains,
the separability of the CDT can be exploited by
transforming the considered measures to a series of line measures 
using the Radon or generalized Radon transform \cite{KolParRoh16,KolNadSimBadRoh19}.
If we replace the Radon transform by
the vertical slice or 
the normalized semicircle transform,
the procedure can be immediately transferred to measures on the sphere.

\begin{table}[htb]
    \caption{Created datasets to study the distinctiveness of linear SVMs
        with respect to the CDT of the vertical slice/semicircle transform
        respectively. 
        Each datum consists of a single 
        or the equally weighted mixture of two vMF distributions
        with fixed concentration $\kappa = 50$
        and randomly generated $\zb \eta \in \S^2$.}
    \label{tab:dataset}
    \centering
    \footnotesize
    \begin{tabular}{l p{0.3\linewidth} p{0.54\linewidth}}
        \toprule
        dataset 
        & 1st class 
        & 2nd class
        \\
        \midrule
        \#1 
        & single vMF distributions 
        & mixtures of two vMFs, means with fixed distance $\pi/2$
        \\
        \#2 
        & single vMF distributions
        & mixtures of two vMFs
        \\
        \#3 
        & single vMF distributions
        & mixtures of two vMFs, means mirrored at equatorial plane
        \\
        \#4 
        & single vMF distributions
        & mixtures of two vMFs, means mirrored at $\xi_3$ axis
        \\
        \#5 
        & mixtures of two vMFs, means\newline\hspace*{3.78pt} mirrored at equatorial plane
        & mixtures of two vMFs, means mirrored at $\xi_3$ axis
        \\
        \bottomrule
    \end{tabular}
\end{table}

\begin{figure}[htb]
    \centering
    \includegraphics[width=.32\textwidth]{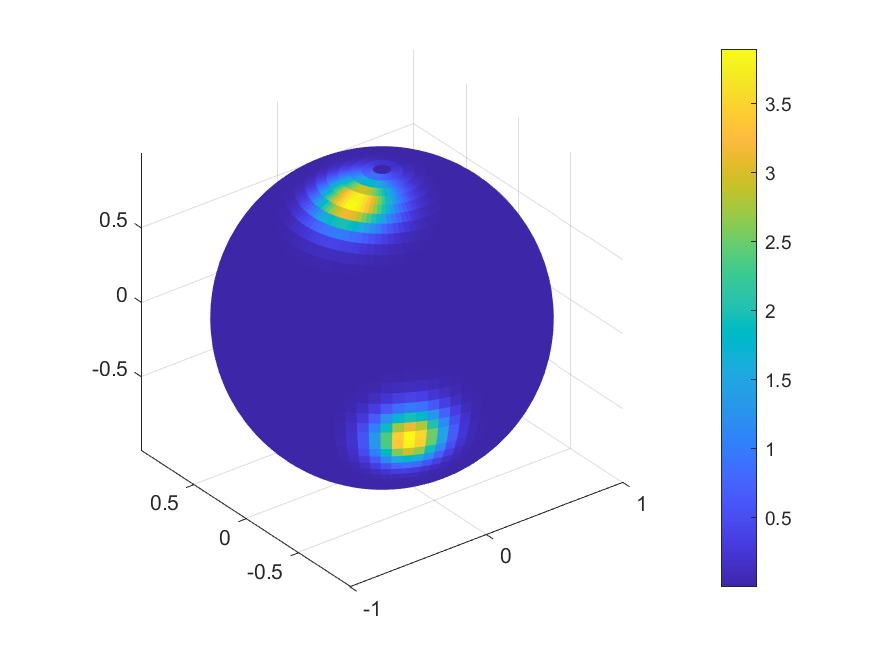}
    \includegraphics[width=.32\textwidth]{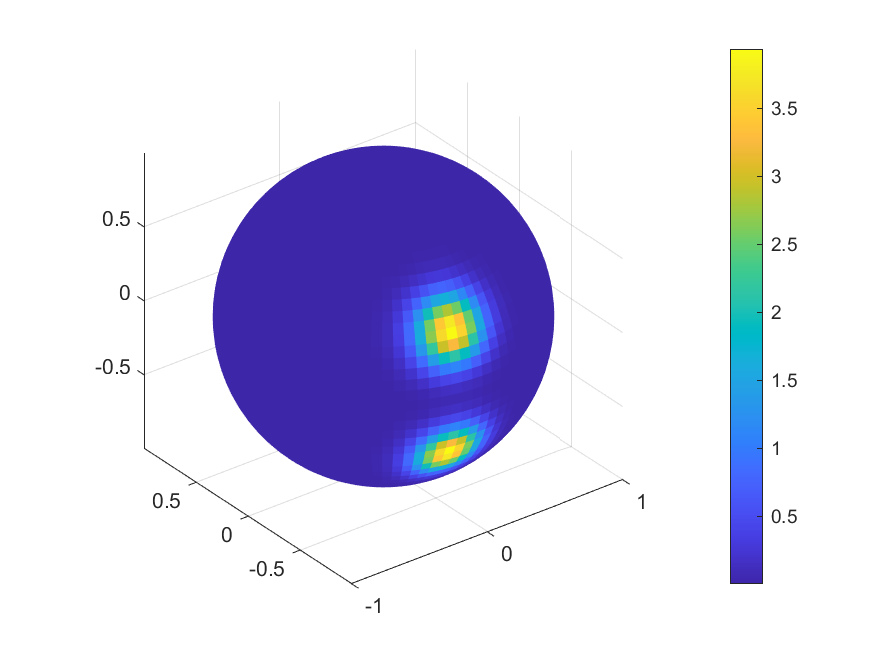}
    \includegraphics[width=.32\textwidth]{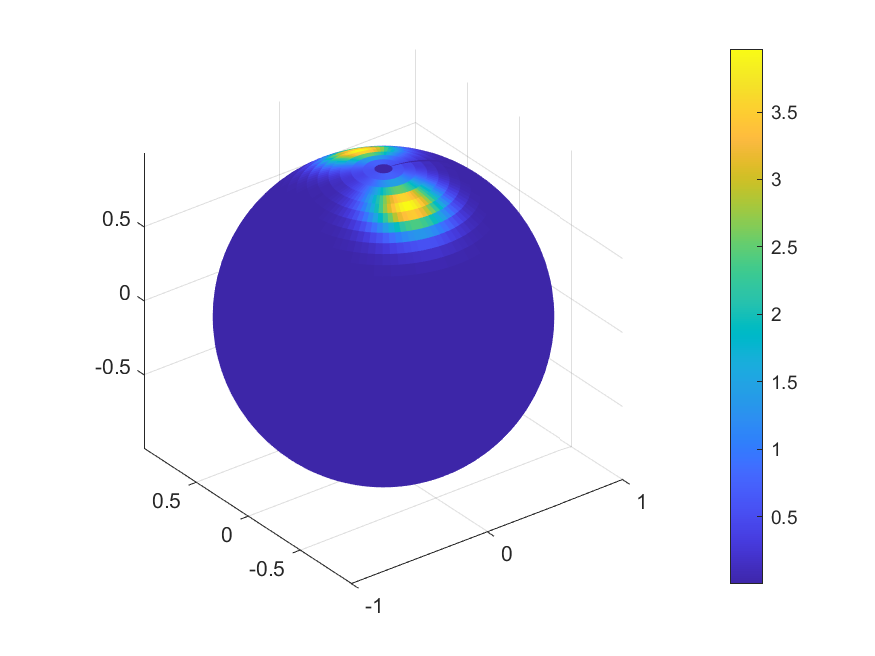}
    \caption{Examples from the generated datasets. From left to right:
    The means are generated with fixed distance $\pi/2$, 
    mirrored at the equator, 
    and mirrored at the $\xi_3$ axis.}
    \label{fig:vMF-mix}
\end{figure}

To show that the vertical slice and semicircle transform
can in principle improve the separability 
between different classes of probability measures,
we built five datasets consisting of 100 measures each.
Each datum represents a (discretized) single 
or mixture density function
of vMF distributions.
The concentration is always chosen as $\kappa = 50$.
The means are randomly generated on $\S^2$
satisfying the restrictions in \autoref{tab:dataset}.
All distributions in the mixtures are equally weighted.
\autoref{fig:vMF-mix} shows some examples of the different classes.
On the basis of these classes,
we train and test linear support vector machines (SVMs)
to study the linear separability after our spherical transformations.
To be more precise,
let $f_\nu$ be the density function of a specific datum. 
This specimen is now transformed into
\begin{equation}
    \bigl( \CDT_{u_\II}[\V_\psi f_\nu] \bigr)_{\psi \in \T}
    \quad\text{and}\quad
    \bigl( \sCDT_{u_\T}[ \W_{\alpha,\beta} f_\nu] \bigr)_{\sph(\alpha,\beta) \in \S^2},
\end{equation}
where $u_\bullet$ denotes the uniform measure.
For the numerical implementation,
we use the quadrature points $\bxi^m$ with $N=44$ from \autoref{sec:int-meas}.

Training and testing of the SVMs is here based on 10-fold cross-validations,
i.e.,
the dataset is divided in 10 subsets
containing equally many samples of each class,
the training is performed on 9 subsets,
and the testing on the remaining. 
The procedure is repeated 10 times 
such that each subset serves one time as testing set.
Before training, the dimension of the training set is reduced to 50
using a principle component analysis.
The success rates of the trained SVMs are given in \autoref{tab:succ-rate}.
Although the vertical slice transform cannot distinguish 
between the upper and lower hemisphere,
the $\V$-CDT approach yields high-quality linear separators
between the classes of single and mixture vMF densities.
The $\V$-CDT approach only fails in experiment \#3,
which is not surprising 
since the samples from the second class are seen
as single vMF densities by $\V$.
The $\W$-CDT approach is useless in \#1 and \#2,
but significantly increases the separability
between single and symmetrized vMF distributions.
These first simulations show that both spherical transforms
can increase the linear separability between certain classes.

\begin{table}
    \caption{Success rates of linear SVMs
    trained and tested directly on the density distributions (---/---),
    the CDT and vertical sliced transformed densities ($\V$-CDT),
    and the $\sCDT$ semicicle transformed densities ($\W$-CDT).
    Mean accuracy and standard deviation are computed 
    with respect to 10-fold cross-validations of the datasets in \autoref{tab:dataset}.}
    \label{tab:succ-rate}
    \centering
    \footnotesize
    \begin{tabular}{lccccc}
         \toprule
         dataset &  \#1 & \#2 & \#3 & \#4 & \#5\\
         \midrule
         ---/--- 
         & $0.555 \pm 0.123$
         & $0.545 \pm 0.145$
         & $0.570 \pm 0.079$
         & $0.585 \pm 0.085$
         & $0.570 \pm 0.136$\\
         $\V$-CDT  
         & $\bm{0.985} \pm 0.024$
         & $\bm{0.985} \pm 0.024$
         & $0.435 \pm 0.111$
         & $\bm{0.995} \pm 0.016$
         & $0.995 \pm 0.016$\\
         $\W$-CDT  
         & $0.465 \pm 0.097$
         & $0.565 \pm 0.125$
         & $\bm{0.865} \pm 0.085$
         & $0.945 \pm 0.037$
         & $\bm{1.000} \pm 0.000$\\
         \bottomrule
    \end{tabular}
\end{table}

\subsubsection*{Acknowledgements}

We thank Sophie Mildenberger for creating an illustration of the semicircle transform.
We gratefully acknowledge the funding by the German Research Foundation (DFG): STE 571/19-1,
project number 495365311, within the SFB F68: \enquote{Tomography Across the Scales}
as well as the BMBF under the project \enquote{VI-Screen} (13N15754).

\appendix

\section{Proof of \autoref{thm:V_inj} } \label{app:thm:V_inj}

Any measure $\mu\in \M_\mathrm{sym}(\S^2)$ is uniquely determined by its application on $C_\mathrm{sym}(\S^2)$,
since, by the definition \eqref{eq:Msym}, we have for any $f\in C(\S^2)$ that
$
\inn{\mu,f}
=
\frac12 \inn{\mu,f+\check f}
$
and $f+\check f \in C_\mathrm{sym}(\S^2)$,
where $\check f(\bxi) = f(\xi_1,\xi_2,-\xi_3)$.
Let $\mu,\nu \in \M_\mathrm{sym}(\S^2)$ such that $\V\mu=\V\nu$. 
Then we obtain for $g\in C(\T\times \II)$ by \autoref{prop:V-as-adj} that
\begin{equation}
  \inn{\mu, \V^*g}
  =
  \inn{\nu, \V^*g}.
\end{equation}
Hence, the claim $\mu=\nu$ holds true if 
$\{\V^*g: g\in C(\T\times \II)\}$ 
is a dense subset of $C_\mathrm{sym}(\S^2)$.
To show this, let $s>2$ and $f\in H_\mathrm{sym}^s(\S^2)$, which is dense in $C_\mathrm{sym}(\S^2)$.
Here we denote by $H_\mathrm{sym}^s(\S^2)$ the subset of even functions of the Sobolev space $H^s(\S^2)$, see \eqref{eq:Hs-norm}.
Since $\V$ is injective by \autoref{prop:V-inj},
we have 
$f = \V^* g$
if and only if
$\V f = \V\V^* g$.
In the following, we show that 
$$
g
\coloneqq
(\V\V^*)^{-1} \V f
$$
is continuous on $\T\times \II$, 
then we obtain $f = \V^* g$, which shows the assertion.
We proceed in a similar manner as for the proof of Sobolev's embedding theorem, cf. \cite[lem.\ 6.14]{Mic13}.

Recall the right singular functions $B_n^k$ of $\V$ from \eqref{eq:Bnk}.
Since $\V^*$ has the same singular functions as $\V$ and the conjugate singular values $\overline{\mathrm{v}_n^k} = \mathrm{v}_n^k$,
we have by \autoref{prop:V-inj}
\begin{equation}
  \V\V^* h 
  =
  \sum_{n=0}^{\infty} \sum_{\substack{k=-n\\n+k\text{ even}}}^{n}
  \abs{\mathrm{v}_n^k}^2\, \, \inn{h, B_n^k}_{L^2(\SO)}\, B_n^k
  ,\qquad \forall h\in L^2(\T\times \II).
\end{equation}
Hence, again by the singular value decomposition of $\V$, we have
\begin{equation} \label{eq:VV}
  (\V \V^*)^{-1} \V f
  =
  \sum_{n=0}^{\infty} \sum_{\substack{k=-n\\n+k\text{ even}}}^{n}
  \frac{1}{\mathrm{v}_n^k}\, \inn{f, Y_n^k}_{L^2(\S^2)}\, B_n^k.
\end{equation}
We want to show that the right hand side of \eqref{eq:VV} converges uniformly on $C(\T\times \II)$.
Let $(\psi,t) \in \T\times \II$.
As the Legendre polynomials satisfy $|P_n(t)|\le 1$ for all $t\in \II$, cf.\ \cite[§ 8.917]{GrRy14}, we have 
\begin{equation}
  \abs{B_n^k(\psi,t)}
  = \abs{\sqrt{\frac{2n+1}{4\pi}}\, P_n(t) \, \e^{\i k \psi}}
  \le \sqrt{\frac{2n+1}{4\pi}} .
\end{equation}
Let $N\in\N$.
Using the bound \eqref{eq:V-sv-bound} on the singular values of $\V$, we see that there exists $C>0$ such that
\begin{align} 
  &\quad
    \abs{\sum_{n=0}^{\infty} \sum_{\substack{k=-n\\n+k\text{ even}}}^{n}
  \frac{1}{\mathrm{v}_n^k}\, \inn{f, Y_n^k}_{L^2(\S^2)}\, B_n^k(\psi,t)
  -\sum_{n=0}^{N-1} \sum_{\substack{k=-n\\n+k\text{ even}}}^{n}
  \frac{1}{\mathrm{v}_n^k}\, \inn{f, Y_n^k}_{L^2(\S^2)}\, B_n^k(\psi,t)
  }
  \\
  &\le
    C \sum_{n=N}^{\infty} \sum_{\substack{k=-n\\n+k\text{ even}}}^{n}
  \left(n+\tfrac12\right)\, \abs{\inn{f, Y_n^k}_{L^2(\S^2)}}
  \\
  &\le
    C \sqrt{\sum_{n=N}^{\infty} \sum_{\substack{k=-n\\n+k\text{ even}}}^{n}
  \left(n+\tfrac12\right)^{2s}\, \abs{\inn{f, Y_n^k}_{L^2(\S^2)}}^2}
  \sqrt{\sum_{n=N}^{\infty} \sum_{{k=-n}}^{n}
  \left(n+\tfrac12\right)^{2-2s} },
\end{align}
where the last line follows by the Cauchy--Schwarz inequality.
In the last equation,
the first root converges to zero for $N\to\infty$ since the Sobolev norm $\norm {f}_{H^s(\S^2)}$ is finite,
and the term under the second root,
\begin{equation}
  \sum_{n=N}^{\infty} \sum_{k=-n}^{n}
  \left(n+\tfrac12\right)^{2-2s}
  =
  4 \sum_{n=N}^{\infty}
  \left(n+\tfrac12\right)^{3-2s},
\end{equation}
also converges since $s>2$.
Hence, the right-hand side of \eqref{eq:VV}, whose summands are continuous themselves,
converges uniformly to a continuous function on $\T\times \II$,
which finally implies that $g$ is continuous.
\hfill $\Box$

\section{Proof of \autoref{thm:svd2}}
\label{sec:W-svd}
1. First we show for $n\in\NN$, $k\in\{-n,\dots,n\}$, and $\bQ\in\SO$ that
\begin{equation} \label{eq:BY}
  \mathcal W Y_n^k (\bQ)
  =
  \sum_{j=-n}^n \lambda_n^j\, \overline{D_n^{k,j}(\bQ)},
\end{equation}
which implies \eqref{eq:B_svd}.
By \eqref{eq:B}, we have
\begin{equation} \label{eq:BYint}
  \mathcal W Y_n^k (\bQ)
  =
  \frac{1}{4\pi}
  \int_{0}^{\pi} Y_n^k(\bQ \sph(0,\vartheta))\,\sin\vartheta \d \vartheta
  \overset{\eqref{eq:Y_D}}=
  {\frac{1}{4\pi}}
  \sum_{j=-n}^n D_n^{j,k}(\bQ^\top) \int_{0}^{\pi} Y_n^j(\sph(0,\vartheta)) \sin\vartheta \d \vartheta.
\end{equation}
Noting that $D_n^{j,k}(\bQ^\top) = \overline{D_n^{-k,-j}(\bQ)}$ by \cite[§ 4.4]{Varsha88} and
performing the substitution $z=\cos\vartheta$, 
we see that \eqref{eq:BY} holds with 
\begin{equation} \label{eq:P_int}
  \lambda_n^j
  =
  \frac{1}{4\pi} \int_{-1}^{1} Y_n^j(\sph(0,\arccos z)) \d z
  \overset{\eqref{eq:Y}}{=}
  \frac{1}{4\pi} \sqrt{\frac{2n+1}{4\pi} \frac{(n-j)!}{(n+j)!}} \int_{-1}^{1} P_n^j(z) \d z.
\end{equation}
If $n=0$, then also $j=0$ and we have $P_0^0 = 1,$
which implies that $\lambda_0^0 = 2 (4\pi)^{-1/2}$.
Let $n\in\N$.
If $j=0$, then $P_n^0$ is the Legendre polynomial of degree $n$
and thus we have $\int_{-1}^1 P_n^0(z) \d z = 0$ for $n\ge1$.
If $n+j$ is odd, then $P_n^j$ is an odd function and hence its integral \eqref{eq:P_int} vanishes.
Let us compute \eqref{eq:P_int} for $n=j \in\NN$.
The substitution $z = \cos\vartheta$ and \eqref{eq:Pnk} yield
\begin{align}
  \int_{-1}^1 P_n^n(z) \d z
  &=
    2(-1)^n (2n-1)!! \int_{0}^1 (1-z^2)^{n/2} \d z
  \\&
  =
  2(-1)^n (2n-1)!! \int_{0}^{\pi/2} (\sin\vartheta)^{n+1} \d\vartheta
  \\& \label{eq:int_Pnn}
  =
  (-1)^n (2n-1)!! \frac{n!!}{(n+1)!!}
  \begin{cases}
    2: & n \text{ even},\\
    {\pi}: & n \text{ odd},
  \end{cases}
\end{align}
where the last equality follows by \cite[§3.621]{GrRy14}.
Let $n+j$ be even and $n\ge2$, $j\ge0$.
We are going to use two recurrence relations from \cite[§8.731]{GrRy14}.
First, we compute the integral of the relation
$$
(n-j) P_{n}^j(z)
=
(z^2-1) \partial_z P_{n-1}^j(z) + n z P_{n-1}^j(z).
$$
Using integration by parts and noting that $ P_{n-1}^j(1) = -P_{n-1}^j(-1)$ yields
\begin{equation} \label{eq:P-recursion1}
  (n-j) \int_{-1}^1 P_{n}^j(z) \d z
  =
  (n-2) \int_{-1}^1 z P_{n-1}^j(z) \d z.
\end{equation}
Second, inserting the integral of the recurrence relation 
$$
(2n-1) z P_{n-1}^j(z)
=
(n-j) P_{n}^j(z) + (n+j-1) P_{n-2}^j(z)
$$
into \eqref{eq:P-recursion1} results in
$$
(n-j) \int_{-1}^1 P_{n}^j(z) \d z
=
\frac{n-2}{2n-1}
\int_{-1}^1 \left( (n-j) P_{n}^j(z) + (n+j-1) P_{n-2}^j(z) \right) \d z
$$
and thus
$$
(n-j) ({n+1}) \int_{-1}^1 P_{n}^j(z) \d z
=
(n-2)(n+j-1) \int_{-1}^1 P_{n-2}^j(z) \d z.
$$
Hence, we obtain by \eqref{eq:int_Pnn} that
\begin{align*}
  \int_{-1}^1 P_{n}^j(z) \d z
  &=
    \frac{(n-2)(n+j-1)}{(n-j)(n+1)} \int_{-1}^1 P_{n-2}^j(z) \d z
  \\&
  =
  \frac{(n-2)!!}{(j-2)!!} \frac{(n+j-1)!!}{(2j-1)!!} \frac{1}{(n-j)!!} \frac{(j+1)!!}{(n+1)!!} \int_{-1}^1 P_{j}^j(z) \d z
  \\&
  =
  (-1)^j \frac{j!! (n-2)!! (n+j-1)!!}{(j-2)!! (n-j)!! (n+1)!!} 
  \begin{cases}
    2: & n \text{ even},\\
    {\pi}: & n \text{ odd}.
  \end{cases}
\end{align*}
Together with \eqref{eq:P_int} and \eqref{eq:Pn-k}, this implies \eqref{eq:B_sv}.
\\[1ex]
2. We show that $\{Z^k_n: n \in \NN, k = -n,\ldots,n\}$ forms an orthonormal system in $L^2(\SO)$.
The orthogonality follows from the orthogonality relation \eqref{eq:D_ortho} of the rotational harmonics and the fact that $Z_n^k$ for different indices $(n,k) \neq (n',k')$ contains disjoint linear combinations of rotational harmonics.
Since $\lambda_n^n\neq0$, this also yields that $\mathrm{w}_n\neq0$ for any $n\in\NN$, and hence the normalization follows by definition of $Z_n^k$.
\pagebreak[1]\\[1ex] 
3. We show the bound \eqref{eq:mu-bound}.
Let $n\in\NN$.
The squared singular values of the operator $\mathcal W$ are
\begin{equation} \label{eq:mu1}
  (\mathrm{w}_n)^2
  =
  \norm{\mathcal W Y_n^k}_{L^2(\S^2)}^2
  =
  {\sum_{j=-n}^n
    \abs{\lambda_n^j}^2 \norm{D_n^{k,j}}^2_{L^2(\SO)} }
  \overset{\eqref{eq:D_ortho}}{=}
  {\sum_{j=1}^n
    \abs{\lambda_n^j}^2 \frac{16\pi^2}{2n+1} },
\end{equation}
where we used the identity $\lvert\lambda_n^j\rvert = \lvert{\lambda_n^{-j}}\rvert$.
We have by \eqref{eq:B_sv} that
\begin{align}
  (\mathrm{w}_n)^2
  &=
    \frac{1}{4\pi}\, 
    \sum_{\substack{j=1\\n+j\text{ even}}}^n
  \frac{(n-j)!}{(n+j)!} 
  \left(\frac{j\, (n-2)!!\, (n+j-1)!!}{ (n-j)!!\, (n+1)!!} \right)^2
  \begin{cases}
    4: & n \text{ even},\\
    {\pi}^2: & n \text{ odd}.
  \end{cases}
  \\
  &=
    \frac{1}{4\pi} 
    \left(\frac{ (n-2)!!}{ (n+1)!!} \right)^2
    \sum_{\substack{j=1\\n+j\text{ even}}}^n
  j^2\,
  \frac{(n-j-1)!!}{(n-j)!!}\, 
  \frac{(n+j-1)!!}{ (n+j)!!} 
  \begin{cases}
    4: & n \text{ even},\\
    {\pi}^2: & n \text{ odd}.
  \end{cases}
\end{align}
We use the fact from \cite[p.\ 9]{HiPoQu18} that 
$$
\frac{(2m-1)!!}{(2m)!!}
= {c_m} ({2m+1})^{-1/2},
\quad
\text{where }
\sqrt{\tfrac{2}{\pi}} \le
c_{m}
\le 1
,\qquad \forall m\in\N.
$$
We perform the proof for the case that $n=2m$ is even, the case of odd $n$ is completely analogous. 
We have 
\begin{align}
  (\mathrm{w}_{2m})^2
  &=
    \frac{1}{\pi} 
    \left(\frac{ (2m-2)!!}{ (2m+1)!!} \right)^2
    \sum_{j=1}^m
    (2j)^2\,
    \frac{(2m-2j-1)!!}{(2m-2j)!!}\, 
    \frac{(2m+2j-1)!!}{ (2m+2j)!!} 
  \\
  &= 
    \frac{ (2m+1)}{\pi\, c_m^2\, (2m+1)^2\,(2m)^2} 
    \sum_{j=1}^m
    \frac{4j^2\, c_{m-j}\, c_{m+j}}{\sqrt{2m-2j+1}\, \sqrt{2m+2j+1}} 
  \\
  &=
    \frac{ 1}{\pi\, c_m^2\, (2m+1)\,(2m)^2} 
    \sum_{j=1}^m
    \frac{4j^2\, c_{m-j}\, c_{m+j}}{\sqrt{(2m+1)^2-(2j)^2}} .
\end{align}
Taking into account the bounds on $c_m$ and noting that the summands increase monotonic with $j$,
we replace the sum by an integral plus the last summand and obtain the upper bound
\begin{align}
  (\mathrm{w}_{2m})^2
  &\le
    \frac{ 1}{8 (2m+1)\,m^2} 
    \left(\int_{0}^{m+1/2}
    \frac{(2x)^2}{\sqrt{(2m+1)^2-(2x)^2}} \d x
    + \frac{4m^2}{\sqrt{4m+1}}\right)
  \\
  &=
    \frac{ 1}{8 (2m+1)\,m^2} 
    \left(\frac{\pi(2m+1)^2}{8} + \frac{4m^2}{\sqrt{4m+1}}\right)
    \in \mathcal O(m^{-1}).
\end{align}
For the lower bound, we analogously see that
\begin{align}
  (\mathrm{w}_{2m})^2
  &\ge
    \frac{ 1}{2\pi^2\, (2m+1)\,m^2} 
    \sum_{j=1}^m
    \frac{4j^2}{\sqrt{(2m+1)^2-(2j)^2}} 
    \\&
    \ge
    \frac{ 1}{2\pi^2\, (2m+1)\,m^2} 
    \int_{1}^m
    \frac{4j^2 \d j}{\sqrt{(2m+1)^2-(2j)^2}}
  \\&=
  \frac{ 1}{2\pi^2\, (2m+1)\,m^2} 
  \left(\frac{(2m+1)^2}{4}  \arcsin\left( \frac{2 m}{2 m+1}\right) -\frac{m}{2} \sqrt{4 m+1} \right)
\end{align}
can be bounded from below by a positive multiple of $m^{-1}$ for $m\to\infty$. \hfill $\Box$

\section{Proof of \autoref{thm:B_inj}} \label{app:abc}
Let $\mu,\nu \in \M(\S^2)$ such that $\mathcal W \mu = \mathcal W \nu$.
By \autoref{prop:W-adj-meas}, we have 
\begin{equation}
  \inn{\mu, \mathcal W^* g}
  =
  \inn{\nu, \mathcal W^* g}
  ,\qquad \forall g\in C(\SO).
\end{equation}
The claim holds if we can show that $\{\mathcal W^*g: g\in C(\SO)\}$ is a dense subset of $C(\S^2)$.
Let $f\in H^s(\S^2)$ with $s>2$, cf.\ \eqref{eq:Hs-norm},
which is dense in $C(\S^2)$, see \cite[p.\ 121]{AtHa12}.
We show that $g\coloneqq (\mathcal W \mathcal W^*)^{-1} \mathcal W f \in C(\SO)$,
which also implies $f = \mathcal W^*g$ by the injectivity of $\mathcal W$.
We proceed analogously to the proof of Sobolev's embedding theorem \cite[p.\ 122]{AtHa12}.
Since $\mathcal W^*$ has the same singular functions as $\mathcal W$ and the singular values $\mathrm{w}_n = \overline{\mathrm{w}_n}$,
\autoref{thm:B_inj} implies 
$$
g
=
\sum_{n\in\NN} \sum_{k=-n}^{n}
\frac{1}{\mathrm{w}_n} \inn{f,Y_n^k} Z_n^k.
$$
Let $\bQ\in\SO$.
We have by \eqref{eq:Z} and the Cauchy--Schwarz inequality
\begin{equation} \label{eq:Z-bound}
  \abs{Z_n^k(\bQ)}^2
  \le
  \abs{\mathrm{w}_n}^{-2}\, 
  \bigg(\sum_{j=-n}^n  \abs{\lambda_n^j} \abs{D_n^{k,j}(\bQ)} \bigg)^2
  \overset{\eqref{eq:mu1}}{\le} 
  \frac{2n+1}{16\pi^2}
  {\sum_{j=-n}^n \abs{D_n^{k,j}(\bQ)}^2}.
\end{equation}
Again by the Cauchy--Schwarz inequality, we obtain
\begin{align}
  \abs{g(\bQ)}^2
  &\le
    \bigg(\sum_{n\in\NN} \sum_{k=-n}^{n}
    \abs{\mathrm{w}_n}^{-1}\, \abs{\inn{f,Y_n^k}}\, 
    \abs{Z_n^{k}(\bQ)}\bigg)^2
  \\
  &\le 
    {\sum_{n\in\NN} \sum_{k=-n}^{n}
    (n+\tfrac12)^{2s}\, \abs{\inn{f,Y_n^k}}^2}
    {\sum_{n\in\NN} \sum_{k=-n}^{n}
    (n+\tfrac12)^{-2s}\, \abs{\mathrm{w}_n}^{-2}\, 
    \abs{Z_n^{k}(\bQ)}^2}
  \\
  &\overset{\eqref{eq:Z-bound}}{\le} 
    \norm{f}_{H^s(\S^2)}^2
    {\sum_{n\in\NN} 
    (n+\tfrac12)^{-2s}\, \abs{\mathrm{w}_n}^{-2}\, 
    \frac{2n+1}{16\pi^2}
    \sum_{k=-n}^{n}\sum_{j=-n}^n \abs{D_n^{k,j}(\bQ)}^2}.
    \label{eq:sum-sobolev}
\end{align}
Using the addition formula 
$
\sum_{j,k=-n}^n 
\lvert{D_n^{k,j}(\bQ)}\rvert^2
=
n+1,
$
see \cite[p.\ 17]{hiediss07},
and the bound \eqref{eq:mu-bound} of $\mathrm{w}_n$,
we see that the last sum converges uniformly in $\bQ$.
Since the basis functions $Z_n^k$ are continuous, this implies the continuity of $g$, which proves the assertion.
\hfill $\Box$


\end{document}